\documentclass[11pt,leqno]{article}
\usepackage{amsmath, amscd, amsthm, amssymb, graphics, xypic, mathrsfs, setspace, fancyhdr, times}
\xyoption{all}
\entrymodifiers={+!!<0pt,\fontdimen22\textfont2>}
\DeclareFontFamily{OT1}{pzc}{}
\DeclareFontShape{OT1}{pzc}{m}{it}%
             {<-> s * [1.195] pzcmi7t}{}
\DeclareMathAlphabet{\mathscr}{OT1}{pzc}%
                                 {m}{it}
\usepackage[colorlinks=true,pagebackref=true]{hyperref} 
\hypersetup{backref}

\newcommand{\tensor}{\otimes}

\newcommand{\Spec}{\operatorname{Spec}}

\newcommand{\isomto}{{\stackrel{\sim}{\;\longrightarrow\;}}}
\newcommand{\isomt}{{\stackrel{{\scriptscriptstyle{\sim}}}{\;\rightarrow\;}}}

\newcommand{\coker}{\operatorname{coker}}
\newcommand{\hooklongrightarrow}{\lhook\joinrel\longrightarrow}
\newcommand{\longrightrightarrow}{\relbar\joinrel\twoheadrightarrow}

\renewcommand{\O}{{\mathcal O}}
\renewcommand{\hom}{\operatorname{Hom}}

\newcommand{\cplx}{{\mathbb C}}

\newcommand{\Z}{{\mathbb Z}}
\newcommand{\aone}{{\mathbb A}^1}
\newcommand{\pone}{{\mathbb P}^1}

\newcommand{\gm}{{{\mathbf G}_{m}}}
\renewcommand{\L}{{\mathcal L}}
\newcommand{\N}{{\mathcal N}}
\newcommand{\MW}{\mathrm{MW}}

\newcommand{\et}{\text{\'et}}
\newcommand{\ho}[1]{\mathscr{H}({#1})}
\newcommand{\hop}[1]{\mathscr{H}_{\bullet}({#1})}

\newcommand{\bpi}{\boldsymbol{\pi}}

\newcommand{\Nis}{\operatorname{Nis}}

\newcommand{\Shv}{{\mathscr{Shv}}}
\newcommand{\Sm}{\mathscr{Sm}}

\newcommand{\Spc}{\mathscr{Spc}}

\newcommand{\K}{{{\mathbf K}}}

\newcommand{\hsnis}{\mathscr{H}_s^{\Nis}(k)}
\newcommand{\hspnis}{\mathscr{H}_{s,\bullet}^{\Nis}(k)}

\newcommand{\F}{{\mathcal F}}

\newcounter{intro}
\theoremstyle{plain}
\newtheorem{thm}{Theorem}[section]

\newtheorem{lem}[thm]{Lemma}
\newtheorem{cor}[thm]{Corollary}
\newtheorem{prop}[thm]{Proposition}
\newtheorem*{claim*}{Claim}  

\newtheorem*{thm*}{Theorem}
\newtheorem*{problem*}{Problem}

\newtheorem{thmintro}{Theorem}

\theoremstyle{definition}
\newtheorem{defn}[thm]{Definition}

\newtheorem{notation}[thm]{Notation}

\theoremstyle{remark}
\newtheorem{rem}[thm]{Remark}
\newtheorem{remintro}[thmintro]{Remark}

\numberwithin{equation}{section}

\begin{document}
\pagestyle{fancy}
\renewcommand{\sectionmark}[1]{\markright{\thesection\ #1}}
\fancyhead{}
\fancyhead[LO,R]{\bfseries\footnotesize\thepage}
\fancyhead[LE]{\bfseries\footnotesize\rightmark}
\fancyhead[RO]{\bfseries\footnotesize\rightmark}
\chead[]{}
\cfoot[]{}
\setlength{\headheight}{1cm}

\author{\begin{small}Aravind Asok\thanks{Aravind Asok was partially supported by National Science Foundation Awards DMS-0900813 and DMS-0966589.}\end{small} \\ \begin{footnotesize}Department of Mathematics\end{footnotesize} \\ \begin{footnotesize}University of Southern California\end{footnotesize} \\ \begin{footnotesize}Los Angeles, CA 90089-2532 \end{footnotesize} \\ \begin{footnotesize}\url{asok@usc.edu}\end{footnotesize} \and \begin{small}Jean Fasel\thanks{Jean Fasel was supported by the Swiss National Science Foundation, grant PAOOP2\_129089}\end{small} \\ \begin{footnotesize}Fakult\"at Mathematik\end{footnotesize} \\ \begin{footnotesize}Universit\"at Duisburg-Essen, Campus Essen\end{footnotesize} \\ \begin{footnotesize}Thea-Leymann-Strasse 9, D-45127 Essen\end{footnotesize} \\ \begin{footnotesize}\url{jean.fasel@gmail.com}\end{footnotesize}}

\title{{\bf A cohomological classification of \\ vector bundles on smooth affine threefolds}}
\date{}
\maketitle

\begin{abstract}
We give a cohomological classification of vector bundles of rank $2$ on a smooth affine threefold over an algebraically closed field having characteristic unequal to $2$.  As a consequence we deduce that cancellation holds for rank $2$ vector bundles on such varieties.  The proofs of these results involve three main ingredients.  First, we give a description of the first non-stable $\aone$-homotopy sheaf of the symplectic group.  Second, these computations can be used in concert with F. Morel's $\aone$-homotopy classification of vector bundles on smooth affine schemes and obstruction theoretic techniques (stemming from a version of the Postnikov tower in $\aone$-homotopy theory) to reduce the classification results to cohomology vanishing statements.  Third, we prove the required vanishing statements.
\end{abstract}

\begin{footnotesize}
\tableofcontents
\end{footnotesize}
\newpage

\section{Introduction}
This paper, together with its companions \cite{AsokFaselSpheres,AsokFaselpi3A3minus0}, studies questions in the theory of projective modules using the Morel-Voevodsky $\aone$-homotopy theory \cite{MV}.  The jumping off point of this approach is F. Morel's $\aone$-homotopy classification of vector bundles over smooth affine schemes \cite{MField}, which can be viewed as an algebro-geometric analog of Steenrod's homotopy classification of topological vector bundles \cite[\S 19.3]{Steenrod}.  The basic idea is that, by applying the machinery of obstruction theory in the context of $\aone$-homotopy theory, one may mimic results about classification of topological vector bundles (see, e.g., \cite{DoldWhitney}).  Write ${\mathscr V}_n(X)$ for the set of isomorphism classes of rank $n$ vector bundles on $X$.

\begin{thmintro}[{See Theorem \ref{thm:classificationrank2algclosed} and Corollary \ref{cor:stableisomorphismimpliesisomorphism}}]
\label{thmintro:mainclassification}
Suppose $X$ is a smooth affine $3$-fold over an algebraically closed field $k$ having characteristic unequal to $2$.  The map assigning to a rank $2$ vector bundle $\mathcal{E}$ on $X$ the pair $(c_1({\mathcal{E}}),c_2(\mathcal{E}))$ of Chern classes gives a pointed bijection
\[
\mathcal{V}_2(X) \isomto Pic(X) \times CH^2(X).
\]
\end{thmintro}

For context, recall that  N. Mohan Kumar and M.P. Murthy proved \cite{KumarMurthy} the existence of a rank $2$ vector bundle on a smooth affine threefold over an algebraically closed field having given Chern classes, i.e., the map in the statement was known to be surjective.  Mohan Kumar and Murthy also proved that the map $(c_1,c_2,c_3): \mathcal{V}_3(X) \to \prod_{i=1}^3 CH^i(X)$, analogous to the one studied above, is bijective.

\begin{remintro}
Let $X$ be a smooth affine $3$-fold over an algebraically closed field.  Theorem \ref{thmintro:mainclassification} implies that cancellation holds for vector bundles of rank $2$ on $X$, i.e., if $\mathcal{E}_1,\mathcal{E}_2 \in \mathcal{V}_2(X)$ satisfy $\mathcal{E}_1 \oplus \mathcal{E} \cong \mathcal{E}_2 \oplus \mathcal{E}$ for some vector bundle $\mathcal{E}$, then $\mathcal{E}_1 \cong \mathcal{E}_2$.  A weaker version of this cancellation statement, with $\mathcal{E}_1$ free, was established by the second author in \cite{Fasel3fold}.  Note also that, for vector bundles of rank $\neq 2$, the cancellation statement above is classical: it follows from results of Suslin (for rank $3$ bundles) \cite[Theorem 1]{Suslin77} and Bass-Schanuel \cite[Theorem 2]{Bass62} (for bundles of rank $> 3$).
\end{remintro}

As mentioned before, we use techniques of obstruction theory to describe the set of $\aone$-homotopy classes of maps to an appropriate Grassmannian.  There is a version of the Postnikov tower in $\aone$-homotopy theory, and the main impediments to applying this method to describe homotopy classes arise from our limited knowledge of $\aone$-homotopy sheaves.  Classical homotopy groups are notoriously difficult to compute, and $\aone$-homotopy groups are no different in this respect.  For the result above, we will need information about $\aone$-homotopy groups of $BGL_2$.  By means of the exceptional isomorphism $SL_2 \cong Sp_2$, the computation can be reduced to that of describing the first ``non-stable" $\aone$-homotopy sheaf of the symplectic group.

\begin{thmintro}[See Theorem \ref{thm:firstnonstablehomotopysheafofsp2n} and Lemma \ref{lem:contractionsofsprimeneven}]
\label{thmintro:maincomputation}
If $k$ is a perfect field having characteristic unequal to $2$, there is a canonical short exact sequence of strictly $\aone$-invariant sheaves
\[
0 \longrightarrow {\mathbf T}_4' \longrightarrow \bpi_2^{\aone}(SL_2) \longrightarrow \K^{Sp}_3 \longrightarrow 0.
\]
Here, $\K^{Sp}_3$ is the sheafification of the third symplectic K-theory group for the Nisnevich topology.  The sheaf ${\mathbf T}_4^{\prime}$ sits in an exact sequence of the form
\[
\mathbf{I}^5 \longrightarrow \mathbf{T}'_4 \longrightarrow \mathbf{S}'_4 \longrightarrow 0
\]
where ${\mathbf I}^5$ is the unramified sheaf associated with the $5$th power of the fundamental ideal in the Witt ring, and there is an epimorphism $\K^M_4/12 \to {\mathbf S}_4^{\prime}$ that becomes an isomorphism after $2$-fold contraction.
\end{thmintro}

\begin{remintro}
Wendt \cite{WendtChevalley} provides a rather general description of sections of $\aone$-homotopy sheaves of $SL_n$ (and, more generally, Chevalley groups) as ``unstable Karoubi-Villamayor K-theory."  Our approach complements this approach by a description that is amenable to cohomological computation.
\end{remintro}

The proof of Theorem \ref{thmintro:maincomputation} relies on the theory of $\aone$-fiber sequences attached to $Sp_{2n}$-bundles developed by Morel \cite{MField} and Wendt \cite{WendtTorsor}, which is reviewed in Section \ref{section:aonehomotopystablerange} and complemented in Section \ref{section:aonehomotopynonstable}.  We use some results from the theory of Grothendieck-Witt groups, which is reviewed in Section \ref{section:gwgroups}.

Given Theorem \ref{thmintro:maincomputation}, the general techniques of obstruction theory, as discussed in Section \ref{section:obstructionsandclassification}, reduce the proof of Theorem \ref{thmintro:mainclassification} to identifications of and vanishing statements for certain cohomology groups with coefficients in $\aone$-homotopy sheaves; these vanishing statements, together with some complements, occupy Sections \ref{section:gwgroups} and \ref{section:vanishing}. In particular, we establish a new vanishing theorem for certain $\mathbf{I}^j$-cohomology groups on smooth affine schemes and a precise description of $H^3_{\Nis}(X,\K^{Sp}_3)$ (and its twisted versions) for any smooth scheme $X$ using Chow groups mod $2$ and Steenrod operations (see Theorem \ref{thm:cohomologyofksp3}). Finally, Section \ref{section:complexrealization} discusses compatibility of our computations with complex realization.

\subsubsection*{Acknowledgements}
This project was initially begun when both authors were attending the ``Summer school on rigidity and the conjecture of Friedlander and Milnor" held in August 2011 at the University of Regensburg.  A first version of this paper was put on the ArXiv in May 2012, but the current version differs rather substantially from the initial version both in scope and content; this version owes a debt to many people.  First and foremost, the authors would like to thank Fabien Morel for useful discussion of material in and around that of \cite{MField}; the results contained therein played a formative role in the ideas of this paper.  We have had opportunity to discuss these ideas and results with many people including Matthias Wendt, Chuck Weibel, V. Srinivas, Ben Williams, and Christian Haesemeyer, each of whom contributed useful comments and questions and helped clarify our thoughts.  The first author would like to thank Sasha Merkurjev for explanations about the motivic spectral sequence, the product of which, however, no longer appears in this work.  The second author would like to thank Marco Schlichting for kindly providing the proof of Lemma \ref{lem:multiplicativeaction}.  Finally, the authors thank the referees of previous versions of this paper for a number of comments that helped to streamline and clarify the discussion, and for weeding out some silly mistakes.

\section{$\aone$-fiber sequences and the stable range for $Sp_{2n}$}
\label{section:aonehomotopystablerange}
In this section, we review some preliminaries from $\aone$-homotopy theory, especially some results from the theory of $\aone$-fiber sequences, due to Morel and Wendt, and Morel's classification theorem for vector bundles over smooth affine schemes.  We then recall some stabilization results for $\aone$-homotopy sheaves of linear and symplectic groups; these results are also due to Morel and Wendt.  The ultimate goal of this section is to define the stable range for $Sp_{2n}$, and understand the $\aone$-homotopy sheaves of $Sp_{2n}$ in this range.

\subsubsection*{Preliminaries from $\aone$-homotopy theory}
Assume $k$ is a field.  Write $\Sm_k$ for the category of schemes that are smooth, separated and have finite type over $\Spec k$.  Set $\Spc_k := \Delta^{\circ}\Shv_{\Nis}(\Sm_k)$ (resp. $\Spc_{k,\bullet}$) for the category of (pointed) simplicial sheaves on the site of smooth schemes equipped with the Nisnevich topology; objects of this category will be referred to as {\em (pointed) $k$-spaces}, or simply as {\em (pointed) spaces} if $k$ is clear from the context.  Write $\hsnis$ (resp $\hspnis$) for the (pointed) Nisnevich simplicial homotopy category: this category can be obtained as the homotopy category of, e.g., the injective local model structure on $\Spc_k$ (see, e.g., \cite{MV} for details).  Write $\ho{k}$ (resp. $\hop{k}$) for the associated $\aone$-homotopy category, which is constructed as a Bousfield localization of $\hsnis$ (resp. $\hspnis$).

Given two (pointed) spaces $\mathscr{X}$ and $\mathscr{Y}$, we set $[\mathscr{X},\mathscr{Y}]_{s} := \hom_{\hsnis}(\mathscr{X},\mathscr{Y})$ and $[\mathscr{X},\mathscr{Y}]_{\aone} := \hom_{\ho{k}}(\mathscr{X},\mathscr{Y})$; morphisms in pointed homotopy categories will be denoted similarly with base-points explicitly written if it is not clear from context.  We write $S^i_s$ for the constant sheaf on $\Sm_k$ associated with the simplicial $i$-sphere, and $\gm$ will always be pointed by $1$.  If $\mathscr{X}$ is any space, the sheaf of $\aone$-connected components $\bpi_0^{\aone}({\mathscr X})$ is the Nisnevich sheaf associated with the presheaf $U \mapsto [U,\mathscr{X}]_{\aone}$.  More generally, the $\aone$-homotopy sheaves of a pointed space $(\mathscr{X},x)$, denoted $\bpi_i^{\aone}(\mathscr{X},x)$ are defined as the Nisnevich sheaves associated with the presheaves $U \mapsto [S^i_s \wedge U_+,(\mathscr{X},x)]_{\aone}$.  We also write $\bpi_{i,j}^{\aone}({\mathscr X},x)$ for the Nisnevich sheafification of the presheaf $U \mapsto [S^i_s \wedge \gm^{\wedge j} \wedge U_+,({\mathscr X},x)]_{\aone}$.  For $i = j = 0$, the morphism of sheaves $\bpi_{0,0}^{\aone}(\mathscr{X},x) \to \bpi_0^{\aone}({\mathscr X})$ obtained by ``forgetting the base-point" is an isomorphism.  If $\mathscr{X}$ is an $\aone$-connected space or if the base-point is clear from context, to ease the notational burden, we will sometimes suppress the base-point.

\subsubsection*{Strictly $\aone$-invariant sheaves}
A presheaf of sets $\F$ on $\Sm_k$ is called {\em $\aone$-invariant} if for any smooth $k$-scheme $U$ the morphism $\F(U) \to \F(U \times \aone)$ induced by pullback along the projection $U \times \aone \to U$ is a bijection.  A Nisnevich sheaf of groups $\mathcal{G}$ is called {\em strongly $\aone$-invariant} if the cohomology presheaves $H^i_{\Nis}(\cdot,\mathcal{G})$ are $\aone$-invariant for $i = 0,1$.  A Nisnevich sheaf of abelian groups ${\mathbf A}$ is called {\em strictly $\aone$-invariant} if the cohomology presheaves $H^i_{\Nis}(\cdot,{\mathbf A})$ are $\aone$-invariant for every $i \geq 0$.  We will repeatedly use the fact that strictly $\aone$-invariant sheaves are unramified in the sense of \cite[Definition 2.1]{MField}.  One consequence of this is, e.g. by \cite[Theorem 2.12]{MField}, that if $f:\mathbf {A}\to \mathbf {A}^\prime$ is a morphism of strictly invariant sheaves then $f$ is an isomorphism (or epimorphism or monomorphism) if and only if the induced map on sections over (finitely generated) extensions of the base field are isomorphisms.

\subsubsection*{Contractions}
Suppose $\mathcal{G}$ is a strongly $\aone$-invariant sheaf of groups.  For any smooth $k$-scheme $U$, the unit of $\gm$ defines a morphism $U \to \gm \times U$.  Recall that $\mathcal{G}_{-1}$ is the sheaf
\[
\mathcal{G}_{-1}(U) = \ker(\mathcal{G}(\gm \times U) \to \mathcal{G}(U))
\]
Iterating this construction one defines $\mathcal{G}_{-i}$. If we restrict $()_{-1}$ to the category of strictly $\aone$-invariant sheaves of groups, it is an exact functor (see, e.g., \cite[Lemma 7.33]{MField} or, more precisely, its proof).

\subsubsection*{The Rost-Schmid complex, Gersten resolutions and line bundle twists}
Associated with any strictly $\aone$-invariant sheaf $\mathbf{A}$ is a Rost-Schmid complex \cite[\S 4]{MField}, which coincides with the usual Gersten complex \cite[Corollary 5.44]{MField} and has terms described using contractions of $\mathbf{A}$ \cite[Definition 5.7]{MField}.  This complex provides a flasque resolution of $\mathbf{A}$ by \cite[Theorem 5.41]{MField}, and its cohomology groups compute Nisnevich (or Zariski) cohomology groups of $\mathbf{A}$ \cite[Corollary 5.43]{MField}.

Write $\K^{MW}_n$ for the $n$-th unramified Milnor-Witt K-theory sheaf described in \cite[\S 3.2]{MField}.  As explained in \cite[pp.77-80]{MField}, any strictly $\aone$-invariant sheaf of the form ${\mathbf A}_{-1}$ admits a canonical action of $\gm$ that factors through an action of $\K^{MW}_0$ \cite[Lemma 3.49]{MField}.  Moreover, the $\gm$-action on $\K^{MW}_n$ that arises in this fashion is precisely the one coming from the product map $\K^{MW}_0 \times \K^{MW}_n \to \K^{MW}_n$.

If $X$ is a smooth scheme, $\L$ is a line bundle on $X$, and ${\mathbf A}_{-1}$ is a strictly $\aone$-invariant sheaf, the $\gm$-action just mentioned allows one to ``twist" the sheaf ${\mathbf A}_{-1}$ by $\L$ as follows.  Write $\L^{\circ}$ for the $\gm$-torsor corresponding to the complement of the zero section in $\L$ and define ${\mathbf A}_{-1}(\L)$ as the sheaf on the small Nisnevich site of $X$ associated with the presheaf assigning to an \'etale morphism $u: U \to X$ the abelian group ${\mathbf A}_{-1}(U) \otimes_{\Z[\O(U)^{\times}]} \Z[(u^*\L)^{\circ}(U)]$. In this case, the twisted Rost-Schmid complex $C_*^{RS}(X,\L;{\mathbf A}_{-1})$ defined in \cite[Remark 5.14.2]{MField} provides a flasque resolution of ${\mathbf A}_{-1}(\L)$.

In the special case where ${\mathbf A} = \mathbf{I}^{j}$ the unramified sheaf associated with the fundamental ideal in the Witt ring (see \cite[\S 2.1]{Fasel09d}), the description of the $\K^{MW}_0$-action on ${\mathbf A}_{-1}=\mathbf{I}^{j-1}$ from \cite[Lemma 3.49]{MField} shows that $\mathbf{I}^j(\L)$ and the associated twisted Rost-Schmid complex coincide with the twisted sheaf denoted $I^j_{\L}$ and the twisted Gersten-Witt resolution from \cite{Fasel09d}.  Likewise, the twisted Rost-Schmid complex of $\K^{MW}_n(\L)$ coincides with the fiber product of \cite[Definition 10.2.7]{FaselChowWitt}.  These ``twisted" constructions will only reappear in Sections \ref{section:gwgroups},  \ref{section:vanishing} and \ref{section:obstructionsandclassification}.

\subsubsection*{A review of the theory of $\aone$-fiber sequences}
There is a general definition of a fiber sequence in a model category \cite[\S 6.2]{Hovey}.  This definition generalizes the usual examples coming from Serre fibrations in topology.  We refer the reader to, e.g., \cite[\S 2.3, esp. Definition 2.9]{AsokPi1} for a concise development of these ideas in the context of the $\aone$-homotopy category.  Just like in classical topology, $\aone$-fiber sequences give rise to long exact sequences in $\aone$-homotopy sheaves (see, e.g., \cite[Lemma 2.10]{AsokPi1}).

The main result about fiber sequences we will use is summarized in the following statement, which is quoted from \cite[Proposition 5.1, Proposition 5.2, and Theorem 5.3]{WendtTorsor}; in any situation in this paper where a sequence of spaces is asserted to be an $\aone$-fiber sequence (and for which no auxiliary reference is given), the sequence has this property because of the following result.

\begin{thm}[Morel, Moser, Wendt]
\label{thm:torsor}
Assume $F$ is a field, and $(X,x)$ is a pointed smooth $F$-scheme.  If $P \to X$ is a $G$-torsor for $SL_n, GL_n$ or $Sp_{2n}$, then there is an $\aone$-fiber sequence of the form
\[
 G \longrightarrow P \longrightarrow X.
\]
If, moreover, $Y$ is a pointed smooth quasi-projective $F$-scheme equipped with a left action of $G$, then the associated fiber space, i.e. the quotient $P \times^G Y$, exists as a smooth scheme, and there is an $\aone$-fiber sequence of the form
\[
Y \longrightarrow P \times^G Y \longrightarrow X.
\]
\end{thm}

\begin{proof}[Comments on the proof.]
As regards attribution: Morel proved the above result for $SL_n$ or $GL_n$ ($n = 1$ or $n \geq 3$) in \cite{MField}, and Wendt extended his result to treat a rather general class of reductive groups; the case where $G = SL_2$ requires the results of Moser \cite{Moser}.  In \cite[Proposition 5.1]{WendtTorsor}, the existence of fiber sequences such as the above is stated under the apparently additional hypothesis that $F$ be infinite.  However, the assumption that $F$ is infinite is used by way of Theorem 3.1 of {\em ibid} to guarantee extensibility of $G$-torsors.  In the case of special groups, such as above, this kind of extensibility result is known (e.g., by results of Lindel for vector bundles) without the assumption that the base-field is infinite (see \cite[Remark 3.2]{WendtTorsor}).

Finally, assuming $Y$ is quasi-projective, the quotient $P \times^G Y$ exists as a smooth scheme.  Since $G$ is special (in the sense of Grothendieck-Serre), all $G$-torsors are Zariski (and hence Nisnevich) locally trivial.  Combining these two observations, the notation $P \times^G Y$ is unambiguous, i.e., the Nisnevich sheaf quotient coincides with the scheme quotient.
\end{proof}

\subsubsection*{The stable range}
We now apply the results on fiber sequences above to the $\aone$-fiber sequences
\[
\begin{split}
SL_{n} &\hooklongrightarrow SL_{n+1} \longrightarrow SL_{n+1}/SL_{n} \text{ and }\\
Sp_{2n} &\hooklongrightarrow Sp_{2n+2} \longrightarrow Sp_{2n+2}/Sp_{2n}.
\end{split}
\]
Each of these fiber sequences gives rise to a long exact sequence in $\aone$-homotopy sheaves.  The homogeneous spaces that appear in these fiber sequences are ``highly $\aone$-connected."  More precisely, we have the following result, whose proof we leave to the reader.

\begin{prop}
\label{prop:equivalence}
The ``projection onto the first column" morphism $SL_{n+1} \to {\mathbb A}^{n+1} \setminus 0$ (resp. $Sp_{2n+2} \to {\mathbb A}^{2n+2} \setminus 0$) factors through a morphism $SL_{n+1}/SL_{n} \to {\mathbb A}^{n+1} \setminus 0$ (resp. $Sp_{2n+2}/Sp_{2n} \to {\mathbb A}^{2n+2} \setminus 0$) that is Zariski locally trivial with affine space fibers; in particular, this morphism is an $\aone$-weak equivalence.
\end{prop}

\begin{thm}[{\cite[Theorem 6.40]{MField}}]
\label{thm:morelanminus0}
For any integer $n \geq 1$, the space ${\mathbb A}^{n+1} \setminus 0$ is $(n-1)$-$\aone$-connected, and there is a canonical isomorphism $\bpi_{n}^{\aone}({\mathbb A}^{n+1} \setminus 0) \cong \K^{\MW}_{n+1}$.
\end{thm}

Combining Theorem \ref{thm:morelanminus0} and Proposition \ref{prop:equivalence}, one immediately deduces the following result.

\begin{cor}[Morel, Wendt]
\label{cor:stabilization}
The morphisms $\bpi_i^{\aone}(SL_{n}) \to \bpi_i^{\aone}(SL_{n+1})$ are epimorphisms for $i \leq n-1$ and isomorphisms for $i \leq n-2$.  The morphisms $\bpi_i^{\aone}(Sp_{2n}) \to \bpi_i^{\aone}(Sp_{2n+2})$ are epimorphisms for $i \leq 2n$ and isomorphisms for $i \leq 2n-1$.
\end{cor}

In particular, observe that the sheaves $\bpi_i^{\aone}(SL_n)$ coincide with the stable groups $\bpi_i^{\aone}(SL_{\infty})$ for $i \leq n-2$, and the sheaves $\bpi_i^{\aone}(Sp_{2n})$ coincide with $\bpi_i^{\aone}(Sp_{\infty})$ for $i \leq 2n-1$.  These observations allow us to define the stable range.

\begin{defn}
\label{defn:stablerange}
The sheaves $\bpi_i^{\aone}(SL_n)$ for $i \leq n-2$ and the sheaves $\bpi_i^{\aone}(Sp_{2n})$ for $i \leq 2n-1$ will be said to be {\em in the stable range}.
\end{defn}

\subsubsection*{$\aone$-homotopy sheaves of $Sp_{2n}$ in the stable range}
Using the stabilization results above, it is possible to deduce a stable range for $\aone$-homotopy sheaves of the symplectic group; the corresponding statements for the special linear group were recalled in \cite[Theorem 3.2]{AsokFaselSpheres}.  Given a sheaf of groups $G$, we write $BG$ for its simplicial classifying space in the sense of \cite[\S 4.1]{MV}.

\begin{thm}
\label{thm:stabilizationspn}
For any integers $i,n\in \mathbb{N}$ there are canonical isomorphisms
\[
\bpi_i^{\aone}(Sp_{2n}) \cong \bpi_{i+1}^{\aone}(BSp_{2n}).
\]
If $0 \leq i \leq 2n-1$ and, furthermore, the base field $k$ is assumed to have characteristic unequal to $2$, there are canonical isomorphisms of the form
\[
\bpi_{i+1}^{\aone}(BSp_{2n}) \cong \bpi_{i+1}^{\aone}(BSp_{\infty}) \cong \K^{Sp}_{i+1}.
\]
\end{thm}

\begin{proof}
Since $Sp_{2n}$ is $\aone$-connected for any $n\in\mathbb{N}$, we deduce from \cite[Proposition 8.11]{MField} that $\bpi_i^{\aone}(Sp_{2n})=\bpi_{i+1}^{\aone}(BSp_{2n})$ for any integer $i$.

By Corollary \ref{cor:stabilization}, there are isomorphisms $\bpi_i^{\aone}(Sp_{2n})=\bpi_i^{\aone}(Sp_\infty)$ if $0 \leq i \leq 2n-1$.  Similarly, there are isomorphisms $\bpi_i^{\aone}(Sp_\infty)=\bpi_{i+1}^{\aone}(BSp_\infty)$ and thus $\bpi_{i+1}^{\aone}(BSp_{2n})\cong \bpi_{i+1}^{\aone}(BSp_\infty)$ for $0 \leq i \leq 2n-1$. If $\mathrm{char}(k)\neq 2$, the space $\Z\times BSp_{\infty}$ represents symplectic $K$-theory in $\ho k$ (see for instance \cite[Theorem 8.2]{PaninWalter1}) and it follows that $\bpi_{i+1}^{\aone}(BSp_\infty)=\K^{Sp}_{i+1}$.
\end{proof}

\section{The first non-stable $\aone$-homotopy sheaf of $Sp_{2n}$}
\label{section:aonehomotopynonstable}
The main goal of this section is to describe the first non-stable $\aone$-homotopy sheaf for $Sp_{2n}$ and, as a consequence the ``next" non-stable $\aone$-homotopy sheaf of $GL_2$; this is achieved in Theorem \ref{thm:firstnonstablehomotopysheafofsp2n}.  By the results of Section \ref{section:aonehomotopystablerange}, the first non-stable sheaf is $\bpi_{2n}^{\aone}(Sp_{2n})$, and this sheaf is an extension of a symplectic K-theory sheaf (the stable part) by something ``non-stable".  Our goal here is to describe the non-stable part; we will see that it is an extension of a torsion sheaf by a quotient of a sheafification of a power of the fundamental ideal in the Witt ring.

To begin, observe that the long exact sequence in $\aone$-homotopy sheaves associated with the $\aone$-fiber sequence arising from the $Sp_{2n}$-torsor $Sp_{2n+2} \to Sp_{2n+2}/Sp_{2n}$ yields an exact sequence of the form
\[
\bpi_{2n+1}^{\aone}(Sp_{2n+2}) \stackrel{\varphi_{2n+2}}{\longrightarrow} \bpi_{2n+1}^{\aone}(Sp_{2n+2}/Sp_{2n}) \longrightarrow \bpi_{2n}^{\aone}(Sp_{2n}) \longrightarrow \bpi_{2n}^{\aone}(Sp_{2n+2}) \longrightarrow 0.
\]
The sheaves involving $Sp_{2n+2}$ are already in the stable range by Theorem \ref{thm:stabilizationspn}, and since there is an $\aone$-weak equivalence $Sp_{2n+2}/Sp_{2n} \longrightarrow {\mathbb A}^{2n+2} \setminus 0$, we obtain an exact sequence of the form
\[
\K^{Sp}_{2n+2} \stackrel{\varphi_{2n+2}}{\longrightarrow} \K^{MW}_{2n+2} \longrightarrow \bpi_{2n}^{\aone}(Sp_{2n}) \longrightarrow \K^{Sp}_{2n+1} \longrightarrow 0.
\]
If we set
\[
{\mathbf T}'_{2n+2} := \coker(\K^{Sp}_{2n+2} \stackrel{\varphi_{2n+2}}{\longrightarrow} \K^{MW}_{2n+2}),
\]
then our goal is to describe ${\mathbf T}'_{2n+2}$ explicitly.  To this end, recall from \cite[Proposition 2.6]{AsokFaselSpheres} that there is an exact sequence of sheaves of the form
\[
\xymatrix{0\ar[r] & \mathbf{I}^{2n+3}\ar[r] & \K_{2n+2}^{MW}\ar[r] & \K_{2n+2}^M\ar[r] & 0}
\]
where $\mathbf{I}^{2n+3}$ is the unramified sheaf associated with the $(2n+3)$th power of the fundamental ideal in the Witt ring (again, see \cite[\S 2.1]{Fasel09d}).

Composing $\varphi_{2n+2}$ with the epimorphism $\K^{MW}_{2n+2}\to \K^M_{2n+2}$, we get a morphism $\varphi^\prime_{2n+2}:\K^{Sp}_{2n+2}\longrightarrow \K^M_{2n+2}$ and we set
\[
{\mathbf S}'_{2n+2} := \coker(\K^{Sp}_{2n+2} \stackrel{\varphi^\prime_{2n+2}}{\longrightarrow} \K^{M}_{2n+2}).
\]
The proof of the following lemma is straightforward.

\begin{lem}
\label{lem:comparisonss}
There is an exact sequence of sheaves of the form
\[
\mathbf{I}^{2n+3}\longrightarrow {\mathbf T}'_{2n+2}\longrightarrow {\mathbf S}'_{2n+2}\longrightarrow 0.
\]
\end{lem}

We will see in Section \ref{section:complexrealization} that the morphism $\mathbf{I}^{2n+3}\longrightarrow {\mathbf T}'_{2n+2}$ is, in general, non-trivial. For now, we turn to the task of describing ${\mathbf S}'_{2n+2}$.

The standard inclusion $Sp_{2n+2}\to SL_{2n+2}$ induces a commutative diagram
\[
\xymatrix{Sp_{2n+2}\ar[r]\ar[d] & Sp_{2n+2}/Sp_{2n}\ar[d] \\
SL_{2n+2}\ar[r] & SL_{2n+2}/SL_{2n+1}}
\]
and it is straightforward to check that the right-hand vertical map is an isomorphism.  Moreover, both $Sp_{2n+2}/Sp_{2n}$ and $SL_{2n+2}/SL_{2n+1}$ are $\aone$-weakly equivalent to $\mathbb{A}^{2n+2}\setminus 0$ by Proposition \ref{prop:equivalence}, and the commutative diagram of sheaves obtained by applying $\bpi_{2n+1}^{\aone}(\_)$ to the above diagram takes the form:
\[
\xymatrix{\K_{2n+2}^{Sp}\ar[r]^-{\varphi_{2n+2}}\ar[d] & \K_{2n+2}^{MW}\ar@{=}[d] \\
\bpi_{2n+1}^{\aone}(SL_{2n+2})\ar[r] & \K_{2n+2}^{MW}.}
\]
By \cite[Diagram 3.2]{AsokFaselSpheres}, the map $\bpi_{2n+1}^{\aone}(SL_{2n+2})\to \K_{2n+2}^{MW}$ fits in a commutative diagram
\[
\xymatrix{\bpi_{2n+1}^{\aone}(SL_{2n+2})\ar[r]\ar[d]& \K_{2n+2}^{MW}\ar[d] \\
\K_{2n+2}^Q\ar[r]_-{\psi_{2n+2}} & \K_{2n+2}^M}
\]
where the left vertical arrow is the stabilization map (i.e., the morphism on $\aone$-homotopy sheaves induced by $SL_{2n+2}\to SL_{\infty}$), the right vertical arrow is the canonical epimorphism and $\psi_{2n+2}$ is the sheafified version of a homomorphism originally defined by Suslin in \cite[\S 4]{Suslin82b}.

Combining the two diagrams above, we obtain a diagram of the form
\[
\xymatrix{\K_{2n+2}^{Sp}\ar[r]^-{\varphi_{2n+2}}\ar[d]_-{f_{2n+2,2}} \ar[rd]^-{\varphi^\prime_{2n+2}} & \K_{2n+2}^{MW}\ar[d] \\
\K_{2n+2}^Q\ar[r]_-{\psi_{2n+2}} & \K_{2n+2}^M},
\]
where $f_{2n+2,2}$ is the forgetful homomorphism (which is itself induced by the natural map $BSp_{\infty}\to BSL_{\infty}$; see Section \ref{section:gwgroups} for more information). The image of $\psi_{2n+2}$ is hard to control in general; however, the natural homomorphism from Milnor $K$-theory to Quillen $K$-theory can be sheafified to obtain a morphism of sheaves $\mu_{2n+2}:\K_{2n+2}^M\to \K_{2n+2}^Q$ and the composite
\[
\xymatrix{\K_{2n+2}^M\ar[r]^-{\mu_{2n+2}} & \K_{2n+2}^Q\ar[r]^-{\psi_{2n+2}} & \K_{2n+2}^M}
\]
is multiplication by $(2n+1)!$ (\cite[Lemma 3.10]{AsokFaselSpheres}). Thus $(2n+1)!\K_{2n+2}^M\subset\mathrm{Im}(\psi_{2n+2})$.  As in \cite[Definition 3.6]{AsokFaselSpheres}, we write $\mathbf{S}_{2n+2}$ for the cokernel of $\psi_{2n+2}$.

\begin{prop}\label{prop:commutativeepimorphisms}
There is a commutative diagram of epimorphisms of sheaves of the form
\[
\xymatrix{\K^M_{2n+2}/2(2n+1)!\ar[r]\ar[d] & \mathbf{S}^\prime_{2n+2}\ar[d] \\
\K^M_{2n+2}/(2n+1)!\ar[r] & \mathbf{S}_{2n+2}.}
\]
\end{prop}

\begin{proof}
The maps $GL_n\to Sp_{2n}$ defined by
\[
M\mapsto \begin{pmatrix} M & 0 \\ 0 & (M^{-1})^t\end{pmatrix}
\]
induce a map $BGL_{\infty}\to BSp_{\infty}$ and thus, by taking $\aone$-homotopy sheaves, a morphism of sheaves $H_{2n+2,2}:\K^Q_{2n+2}\to \K_{2n+2}^{Sp}$.

We claim the following diagram commutes:
\[
\xymatrix@C=3em{\K^M_{2n+2}\ar[r]^-{\mu_{2n+2}}\ar@{=}[d] & \K^Q_{2n+2}\ar[r]^-{H_{2n+2,2}}\ar@{=}[d] & \K_{2n+2}^{Sp}\ar[r]^-{\varphi^\prime_{2n+2}}\ar[d]^-{f_{2n+2,2}} & \K^M_{2n+2}\ar@{=}[d] \\
\K^M_{2n+2}\ar[r]_-{\mu_{2n+2}} & \K^Q_{2n+2}\ar[r]_-2 & \K_{2n+2}^{Q}\ar[r]_-{\psi_{2n+2}} & \K^M_{2n+2}.}
\]
Indeed, the middle square commutes by Lemma \ref{lem:computationfH}, while the right-hand square commutes by the discussion preceding the statement of the proposition.   The composite of the bottom maps is equal to multiplication by $2(2n+1)!$, and we obtain the required epimorphism of sheaves $\K^M_{2n+2}/2(2n+1)!\to \mathbf{S}^\prime_{2n+2}$.  The diagram
\[
\xymatrix{\K^M_{2n+2}/2(2n+1)!\ar[r]\ar[d] & \mathbf{S}^\prime_{2n+2}\ar[d] \\
\K^M_{2n+2}/(2n+1)!\ar[r] & \mathbf{S}_{2n+2}}
\]
commutes by construction, and the result follows from Lemma \ref{lem:comparisonss}.
\end{proof}

For convenient reference, we summarize the above results in the following statement.

\begin{thm}
\label{thm:firstnonstablehomotopysheafofsp2n}
There are exact sequences of the form
\[
\begin{split}
0 \longrightarrow {\mathbf T}'_{2n+2} \longrightarrow & \bpi_{2n}^{\aone}(Sp_{2n}) \longrightarrow \K^{Sp}_{2n+1} \longrightarrow 0, \text{ and } \\
0 \longrightarrow \mathbf{D}_{2n+3} \longrightarrow & \mathbf{T}'_{2n+2} \longrightarrow \mathbf{S}'_{2n+2} \longrightarrow 0,
\end{split}
\]
where ${\mathbf T}'_{2n+2}$ and ${\mathbf S}'_{2n+2}$ are defined above, and ${\mathbf S}'_{2n+2}$ is a quotient of $\K^M_{2n+2}/(2(2n+1)!)$, and $\mathbf{D}_{2n+3}$ is a quotient of $\mathbf{I}^{2n+3}$.
\end{thm}

\begin{rem}
In Section \ref{section:complexrealization}, we will complement the result above by establishing Lemmas \ref{lem:contractionsofsprimeneven} and \ref{lem:contractionsofsprimenodd}.  The second result shows that, for $n$ odd, after repeated contraction the epimorphisms $\K^M_{2n+2}/(2(2n+1)!) \to \mathbf{S}'_{2n+2}$ are isomorphisms, while the first result establishes a corresponding result for $n$ even (the statement is a just a bit more complicated). We will also show in Lemma \ref{lem:d2nplus3isnontrivial} that $\mathbf{D}_{2n+3}$ cannot be the zero sheaf for $n$ odd.
\end{rem}

\section{Grothendieck-Witt groups}
\label{section:gwgroups}
In this section, we begin by recalling some basic facts about Grothendieck-Witt groups, which are a Waldhausen-style version of hermitian $K$-theory.  The general reference for the material of this section is the work of M. Schlichting \cite{Schlichting09,Schlichting10}.  The main results are Lemma \ref{lem:computationfH}, which was used in the proof of Proposition \ref{prop:commutativeepimorphisms}, and Theorem \ref{thm:cohomologyofksp3}.  We also spend some time discussing the natural action of $\gm$ on contractions of Grothendieck-Witt sheaves, which provides a technical ingredient in the proof of Theorem \ref{thmintro:mainclassification}.

\subsection*{Definitions}
Let $X$ be a smooth scheme with $2\in \O_X(X)^\times$ (we keep these assumptions throughout the section, though it is not necessary for some of the arguments).  For every such $X$, and any line bundle $\L$ on $X$, one has an {\em exact category with weak equivalences and (strong) duality} $(\mathcal{C}^b(X),qis,\sharp_{\L},\varpi_{\L})$ in the sense of \cite[\S 2.3]{Schlichting10}.  Here, $\mathcal{C}^b(X)$ is the category of bounded complexes of locally free coherent sheaves on $X$, weak equivalences are given by quasi-isomorphisms of complexes, $\sharp_{\L}$ is the duality functor induced by the functor $\hom_{\O_X}(\_,\L)$, and the natural transformation $\varpi_{\L}:1\to \sharp_{\L}\sharp_{\L}$ is induced by the canonical identification of a locally free sheaf with its double dual.  The (left) translation (or shift) functor $T:\mathcal{C}^b(R)\to \mathcal{C}^b(R)$ yields new dualities $\sharp^n_{\L}:=T^n\circ \sharp_{\L}$ and canonical isomorphisms $\varpi^n_{\L}:=(-1)^{n(n+1)/2}\varpi_{\L}$.

Associated with an exact category with weak equivalences and (strong) duality is a Grothendieck-Witt space and higher Grothendieck-Witt groups, obtained as homotopy groups of the Grothendieck-Witt space \cite[\S 2.11]{Schlichting10}.  We write $\mathcal{GW}(\mathcal{C}^b(X),qis,\sharp^j_{\L},\varpi^j_{\L})$ for the Grothendieck-Witt space of the example described in the previous paragraph.

\begin{defn}
\label{defn:gwgroups}
For $i\geq 0$, we denote by $GW^j_i(X,{\L})$ the group $\pi_i\mathcal{GW}(\mathcal{C}^b(X),qis,\sharp^j_{\L},\varpi^j_{\L})$. If ${\L}=\O_X$, we write $GW^j_i(X)$ for $GW^j_i(X,\O_X)$.
\end{defn}

Later, we will also use ``negative Grothendieck-Witt groups:" Schlichting constructs a spectrum $\mathbb{G}W(\mathcal{C}^b(X),qis,\sharp^j_{\L},\varpi^j_{\L})$ \cite[\S 10]{Schlichting10}, and the negative Grothendieck-Witt groups are defined as $GW_{-i}^j(X,{\L}):=\pi_{-i}\mathbb{G}W(\mathcal{C}^b(X),qis,\sharp^j_{\L},\varpi^j_{\L})$ for $i>0$.

For any $j\in\Z$, the group $GW_0^j(X,{\L})$ coincides with the Grothendieck-Witt group defined by Balmer-Walter of the triangulated category $D^b(X)$ of bounded complexes of coherent locally free $\O_X$-modules endowed with the corresponding duality (\cite[Lemma 8.2]{Schlichting08pre}, \cite[Theorem 5.1]{Walter03}), and negative Grothendieck-Witt groups coincide with triangular Witt groups as defined by P. Balmer (see, e.g., \cite{Balmer05b}) under the formula $GW_{-i}^j(X,{\L})=W^{i+j}(X,{\L})$; these identifications will be used when we study ``contractions" in Proposition \ref{prop:contractionofgw}.

The Grothendieck-Witt groups defined above coincide with hermitian $K$-theory as defined by M. Karoubi \cite{Karoubi,Karoubi80} in the case of affine schemes, at least when $2$ is invertible (see \cite[Remark 4.16]{Schlichting09}, see also \cite{Horn}).  In particular, given a smooth $k$-algebra $R$ we have by \cite[Corollary A.2]{Schlichting08pre} the identifications
\[
\begin{split}
GW_i^0(R) &= K_iO(R) \\
GW_i^2(R) &= K_iSp(R).
\end{split}
\]
There are identifications $GW_i^1(R)={}_{-1}U_i(R)$ and $GW_i^3(R)=U_i(R)$, where the groups $U_i(R)$ and ${}_{-1}U_i(R)$ are Karoubi's $U$-groups, and $GW_i^n$ is $4$-periodic in $n$.

One can compare Quillen $K$-theory with higher Grothendieck-Witt groups using the hyperbolic morphisms $H_{i,j}:K_i(X)\to GW_i^j(X,{\L})$ and forgetful morphisms $f_{i,j}:GW_i^j(X,{\L})\to K_i(X)$ defined for any $i,j\in \mathbb{N}$ and any line bundle ${\L}$ over $X$ (note the indexing!). The hyperbolic and forgetful morphisms are connected by means of the \emph{Karoubi periodicity} exact sequences (\cite[Theorem 6.1]{Schlichting08pre})
\[
\xymatrix@C=1.71em{\ldots\ar[r] & K_i(X)\ar[r]^-{H_{i,j}} & GW_i^j(X,{\L})\ar[r]^-{\eta} & GW_{i-1}^{j-1}(X,{\L})\ar[r]^-{f_{i-1,j-1}} & K_{i-1}(X)\ar[r]^-{H_{i-1,j}} & GW_{i-1}^j(X,{\L})\ar[r] & \ldots},
\]
where $\eta$ is multiplication by a distinguished element in $GW_{-1}^{-1}(k)$.

The composition $f_{i,j}\circ H_{i,j}$ is, in general, difficult to understand, but the situation is slightly better when we take $X$ to be (the spectrum of) a field.  For any field $F$, the identification $\mu_1:K^M_1(F)\to K^Q_1(F)$, induces via cup-product a (functorial in $F$) homomorphism $\mu_i:K^M_i(F)\to K^Q_i(F)$.

\begin{lem}\label{lem:computationfH}
For any field $F$ having characteristic unequal $2$, and for any integers $i,j \geq 0$, the following diagram commutes:
\[
\xymatrix@C=6em{K^Q_i(F)\ar[r]^-{f_{i,j}\circ H_{i,j}} & K^Q_i(F) \\
K^M_i(F)\ar[r]_-{(1+(-1)^{i+j})Id}\ar[u]^-{\mu_i} & K^M_i(F).\ar[u]_-{\mu_i}}
\]
\end{lem}

\begin{proof}
Let $(\mathcal E,\omega,\sharp,\eta)$ be an exact category with weak-equivalences and duality in the sense of \cite[\S 2.3]{Schlichting10}. With any exact category with weak-equivalences, one can associate the hyperbolic category $(\mathcal {HE},\omega)$ \cite[\S 2.15]{Schlichting10}. Its objects are pairs $(X,Y)$ of objects of $\mathcal C$, a morphism $(X,Y)\to (X^\prime,Y^\prime)$ is a pair $(a,b)$ of morphisms of $\mathcal C$ with $a:X\to X^\prime$ and $b:Y^\prime\to Y$. Such a morphism is a weak-equivalence if $a$ and $b$ are. The switch $(X,Y)\mapsto (Y,X)$ yields a duality $^\star$ on $\mathcal {HE}$ and there is an obvious identification $id:1\to ^{\star\star}$. Thus $(\mathcal {HE},\omega,^\star,id)$ is an exact category with weak-equivalences and duality. The Grothendieck-Witt space $\mathcal{GW}(\mathcal {HE},\omega,^\star,id)$ is naturally homotopic to the $K$-theory space $\mathcal K(\mathcal E,\omega)$ \cite[Proposition 2.17]{Schlichting10}. In this context, the forgetful functor $f$ reads as $f(X)=(X,X^\sharp)$ for any $X$ in $\mathcal E$. On the other hand the hyperbolic functor $H:\mathcal {HE}\to \mathcal E$ is defined by $H(X,Y)=X\oplus Y^\sharp$ \cite[\S 3.9]{Schlichting09}. The composition $fH:(\mathcal {HE},\omega,^\star,id)\to (\mathcal {HE},\omega,^\star,id)$ is then given by $fH(X,Y)=(X\oplus Y^\sharp, X^\sharp\oplus Y^{\sharp\sharp})$. In particular, this composition is the same for $(\mathcal E,\omega,\sharp,\eta)$ and $(\mathcal E,\omega,\sharp,-\eta)$.

We now specialize to the case of $(\mathcal{C}^b(F),qis,\sharp^j,\varpi^j)$.  By the discussion of the previous paragraph, the map $fH$ on $K$-theory sends an object $X$ of $\mathcal{C}^b(F)$ to $X \oplus X^{\sharp}[j]$.  By the additivity theorem (see, e.g., \cite[Chapter V.1.2]{Weibel}) the induced map $K(F) \to K(F)$ is thus $Id + (-1)^j \tau$, where $\tau$ is induced by the involution of $GL(F)$ defined by $G\mapsto (G^t)^{-1}$.  Now, $\tau$ acts as the identity on $K_0(F)$ and is easily seen to be multiplication by $-1$ on $K_1(F)$.
\end{proof}

\subsubsection*{Representability and Gersten complex}
J. Hornbostel \cite[\S 3]{Hornbostel05} showed that there are spaces $\mathscr{GW}^j$ in $\hop k$ for any $j\in\mathbb{Z}$ such that for any smooth scheme $X$ we have $GW_i^j(X)=[S^i_s\wedge X_+,\mathscr{GW}^j]_{\aone}$ (see also \cite[\S 6]{PaninWalter1}).

\begin{defn}
\label{defn:gwsheaves}
For $i\in\mathbb N$, we set
\[
\begin{split}
\mathbf{GW}^j_i &:= \bpi_i^{\aone}(\mathscr{GW}^j), \text{ and } \\
\mathbf{GW}^j_{-i} &:= \bpi^{\aone}_{0,i}(\mathscr{GW}^{i+j}).
\end{split}
\]
\end{defn}

Since $GW_i^2(X)=K_iSp(R)$ for any ring $R$, we have $\mathbf{GW}_i^2=\K_i^{Sp}$ for any $i\in\mathbb{N}$. Moreover, it follows from \cite[Lemma 5.2]{PaninWalter1} that the sheaves $\mathbf{GW}^j_{-i}$ are the Nisnevich sheaves associated with the presheaves $U\mapsto W^{i+j}(U)$. Witt and Grothendieck-Witt groups are also representable in the stable $\aone$-homotopy category of $\pone$-spectra (\cite[\S 5]{Hornbostel05}), and therefore the sheaves $\mathbf{GW}_i^j$ are strictly $\aone$-invariant for any $i,j\in\mathbb Z$.  Using \cite[Theorem 6.13]{MField} (as well as \cite[Lemma 5.2]{PaninWalter1}), we deduce the following result on contractions of Grothendieck-Witt sheaves.

\begin{prop}
\label{prop:contractionofgw}
For any $i,j\in\mathbb{Z}$, we have $(\mathbf{GW}^j_{i})_{-1}=\mathbf{GW}^{j-1}_{i-1}$.
\end{prop}

\begin{rem}
\label{rem:karoubiperiodicitycontractions}
The proof of \cite[Lemma 5.2]{PaninWalter1} also gives the contractions of the hyperbolic morphisms $H_{i,j}:\K^Q_i\to \mathbf{GW}_i^j$ and forgetful morphisms $f_{i,j}:\mathbf{GW}_i^j\to \K_i^Q$. We find $(H_{i,j})_{-1}=H_{i-1,j-1}$ and $(f_{i,j})_{-1}=f_{i-1,j-1}$.
\end{rem}

\subsubsection*{Some actions of $\gm$ on Grothendieck-Witt sheaves}
Proposition \ref{prop:contractionofgw} implies that, for any pair of integers $i,j$, the sheaf $\mathbf{GW}^j_i$ is a contraction.  As a consequence, the sheaf $\mathbf{GW}^j_i$ is equipped with a natural action of $\gm$ that factors through an action of $\K^{MW}_0 \cong \mathbf{GW}^0_0$; see the beginning of Section \ref{section:aonehomotopystablerange} for a more discussion of this action.  On the other hand, the multiplicative structure on Grothendieck-Witt groups \cite[\S 5.4]{Schlichting08pre} determines a morphism
\[
\mathbf{GW}^0_0 \times \mathbf{GW}^j_i \longrightarrow \mathbf{GW}^j_i.
\]
The next lemma shows that, under the canonical isomorphism between $\K^{MW}_0$ and $\mathbf{GW}^0_0$, these two actions coincide.

\begin{lem}
\label{lem:actionongwij}
The action of $\K^{MW}_0$ on $\mathbf{GW}^j_i$ coming from \textup{\cite[Lemma 3.49]{MField}} coincides with the multiplicative action of $\mathbf{GW}^0_0$ on $\mathbf{GW}^j_i$ under the canonical isomorphism $\K^{MW}_0 \isomt \mathbf{GW}^0_0$.
\end{lem}

\begin{proof}
We unwind the definition of the action of $\K^{MW}_0$ on $\mathbf{GW}^0_0$.  For any smooth scheme $X$, we have an exact sequence of the form
\[
0 \longrightarrow \mathbf{GW}^j_i(X) \longrightarrow \mathbf{GW}^j_{i}(X\times\gm) \stackrel{\pi}{\longrightarrow} \mathbf{GW}^{j-1}_{i-1}(X) \longrightarrow 0.
\]
One can construct a splitting of $\pi$ as follows.  The projection map $p: X \times \gm \to X$ induces a pullback map $\mathbf{GW}^{j-1}_{i-1}(X) \to \mathbf{GW}^{j-1}_{i-1}(X \times \gm)$.  On the other hand, the canonical homomorphism $\K^{MW}_1(\gm) \to \mathbf{GW}^1_1(\gm)$ is an isomorphism (this follows from \cite[Corollaire 4.5.1.5]{BargeLannes} after unwinding the definition of $\mathbf{GW}^1_1$ and using the description of $\K^{MW}_1$ as a fiber product).  Picking a coordinate $t$ on $\gm$, there is, under this identification, a canonical class $[t] \in \mathbf{GW}^1_1$.  The element $[t]$ then determines a class in $\mathbf{GW}^{1}_{1}(X \times \gm)$ by pullback, and multiplication by this element determines a homomorphism $\mathbf{GW}^{j-1}_{i-1}(X \times \gm) \to \mathbf{GW}^j_i(X \times \gm)$ that provides the required splitting.

Now, given any $a \in \O_X(X)^{\times}$, the action on $\mathbf{GW}^{j-1}_{i-1}(X)$ is given by the formula $a \cdot \alpha = \pi(p^*\alpha\cdot [at])$.  Since $\K_1^{MW}(X\times \gm)=\mathbf{GW}_1^1(X\times\gm)$, we have $[at]=[a]+\langle a\rangle [t]$ where $\langle a\rangle$ is the class defined by $a$ in $\mathbf{GW}_0^0(X)$. Since $\pi$ is $\mathbf{GW}_0^0$-linear, we get
\[
\pi(p^*\alpha\cdot [at])=\pi(p^*\alpha\cdot [a])+\pi(p^*\alpha\cdot \langle a\rangle [t])=\langle a\rangle\cdot \pi(p^*\alpha\cdot [t]),
\]
which provides the required identification.
\end{proof}

Using Proposition \ref{prop:contractionofgw}, it follows from \cite[Definition 5.7]{MField} that the degree $n$ piece of the Rost-Schmid complex for $\mathbf{A} = \mathbf{GW}^j_i$ has terms of the form
\[
\bigoplus_{x\in X^{(n)}} GW_{i-n}^{j-n}(k(x),\omega_x^{\L}),
\]
where $\omega_x^{\L}=\mathrm{Ext}^n_{\mathcal O_{X,x}}(k(x),\L\otimes\mathcal O_{X,x})$.

A priori, there is another action of $\gm$ on $\mathbf{GW}_i^2$ for any $i\in\mathbb{N}$ defined as follows.
If $R$ is a smooth $k$-algebra and $t\in R^\times$ is an invertible element, then let $\sigma_{2,t}$ be the invertible matrix $\begin{pmatrix} t & 0 \\ 0 & 1\end{pmatrix}$.  Define $2n \times 2n$-matrices inductively by means of the formula $\sigma_{2n,t}:=\sigma_{2n-2,t}\perp \sigma_{2,t}$.  For any natural number $n$, the function assigning to $t \in R^{\times}$ the operation of conjugating a matrix $X \in Sp_{2n}(R)$ by $\sigma_{2n,t}$ defines an action of $R^{\times}$ on $Sp_{2n}(R)$.  This construction defines an action of $\gm$ on $Sp_{2n}$ and consequently an action of $\gm$ on $BSp_{2n}$.  The stabilization morphisms $BSp_{2n} \to BSp_{2n+2}$ are, by construction, equivariant for the $\gm$-actions so defined.  As a consequence, there is an induced action
\begin{equation}
\label{eqn:action}
\gm\times \mathbf{GW}_i^{2} \longrightarrow \mathbf{GW}_i^{2}.
\end{equation}
We can compare this action of $\gm$ to the one discussed above.

\begin{lem}\label{lem:multiplicativeaction}
For any smooth $k$-algebra $R$, and any $t \in R^{\times}$, the action of $t$ on $\mathbf{GW}^2_i(R)$ described in \textup{\ref{eqn:action}} is given by multiplication by $\langle t \rangle$.
\end{lem}

\begin{proof}
Let $R$ be a smooth algebra as above, and let $t\in R^\times$. Let $\mathcal S$ be the symmetric monoidal category with  objects given by pairs $(P,\varphi)$ where $P$ is a (finitely generated) projective $R$-module equipped with a non-degenerate skew-symmetric form $\varphi$ and with morphisms given by isometries.  An element $t \in R^{\times}$ determines a (strict) symmetric monoidal functor $\cdot \tensor \langle t \rangle: \mathcal{S} \to \mathcal{S}$ as follows:  on objects $(P,\varphi) \tensor \langle t \rangle := (P,t\varphi)$, and act by the identity on morphisms; this specifies a categorical action of $R^{\times}$ on $\mathcal{S}$.

The groups $GW^2_i(R)$ can be identified \cite[Theorem A.1]{Schlichting08pre} as homotopy groups of $B \mathcal{S}^{-1}\mathcal{S}$ (here $\mathcal{S}^{-1}\mathcal{S}$ is Quillen's construction; see, e.g., \cite[Chapter IV Definitions 4.2-3]{Weibel}).  The categorical action of $R^{\times}$ described above defines an action of $R^{\times}$ on the space $B \mathcal{S}^{-1}\mathcal{S}$.  The action of $R^{\times}$ on the connected component of the base-point $(\mathcal{S}^{-1}\mathcal{S})_0$ of $B \mathcal{S}^{-1}\mathcal{S}$ is given as follows: if $h_{2n}$ is the usual hyperbolic space of rank $2n$, then $t \cdot (h_{2n},h_{2n}) = (h_{2n} \tensor \langle t \rangle,h_{2n} \tensor \langle t \rangle)$ and $(f,g) \mapsto (f,g)$.

Now, the map
\[
c:BSp\to (\mathcal S^{-1}\mathcal S)_0
\]
is the map $\mathrm{hocolim}_{n\in\mathbb N}\mathrm{Aut}(h_{2n})\to (\mathcal S^{-1}\mathcal S)_0$ defined on objects by $n\mapsto (h_{2n},h_{2n})$ and on morphisms by $f\mapsto (1,f)$ (where $f\in\mathrm{Aut}(h_{2n}$)). Now, consider the diagram
\begin{equation}\label{eqn:homotopycommutative}
\xymatrix{\mathrm{hocolim}_{n\in\mathbb N}\mathrm{Aut}(h_{2n})\ar[r]\ar[d]_-c & \mathrm{hocolim}_{n\in\mathbb N}\mathrm{Aut}(h_{2n})\ar[d]^-c \\
(\mathcal S^{-1}\mathcal S)_0\ar[r] & (\mathcal S^{-1}\mathcal S)_0}
\end{equation}
where the top map is induced by the conjugation by $\sigma_{2n,t}$ and the bottom map is the action of $t$ on $(\mathcal S^{-1}\mathcal S)_0$ described above. This diagram is not strictly commutative. Indeed, one composite maps an object $n$ to $(h_{2n} \tensor \langle t \rangle,h_{2n} \tensor \langle t \rangle)$ and a morphism $f$ to $(1,f)$, while the other composite maps $n$ to $(h_{2n},h_{2n})$ and $f$ to $(1,\sigma_{2n,t}^{-1}f\sigma_{2n,t})$. However, there is a natural isomorphism given by $(\sigma_{2n,t},\sigma_{2n,t})$ as shown by the following diagram
\[
\xymatrix@C=5.5em{(h_{2n},h_{2n})\ar[r]^-{(1,\sigma_{2n,t}^{-1}f\sigma_{2n,t})}\ar[d]_-{(\sigma_{2n,t},\sigma_{2n,t})} & (h_{2n},h_{2n})\ar[d]^-{(\sigma_{2n,t},\sigma_{2n,t})} \\
(t\cdot h_{2n},t\cdot h_{2n})\ar[r]_-{(1,f)} & (t\cdot h_{2n},t\cdot h_{2n}).}
\]
It follows that the diagram (\ref{eqn:homotopycommutative}) commutes up to homotopy and the result stands proved.
\end{proof}

As explained before the lemma, the action of $\gm$ on $Sp_{2n}$ stabilizes. As a consequence, the homogeneous spaces $Sp_{2n+2}/Sp_{2n}$ inherit an action of $\gm$ such that the quotient map $Sp_{2n+2}\to Sp_{2n+2}/Sp_{2n}$ is equivariant.   The projection onto the first row map $Sp_{2n+2} \to {\mathbb A}^{2n+2}\setminus 0$ factors through a projection map $Sp_{2n+2}/Sp_{2n}\to \mathbb{A}^{2n+2}\setminus 0$ that is an $\aone$-weak equivalence.  If we equip ${\mathbb A}^{2n+2}\setminus 0$ with the $\gm$-action given (functorially) by
\[
t \cdot (a_1,\ldots,a_{2n+2})) = (a_1,t^{-1}a_2,a_3,t^{-1}a_4,\ldots,a_{2n+1},t^{-1}a_{2n+2})
\]
the map $Sp_{2n+2}/Sp_{2n} \to \mathbb{A}^{2n+2}\setminus 0$ becomes $\gm$-equivariant.

\begin{lem}\label{lem:trivialaction}
If $n \in {\mathbb N}$ is even, then the $\gm$-action on ${\mathbb A}^{2n}\setminus 0$ described above is $\aone$-homotopically trivial.
\end{lem}

\begin{proof}
For any (finitely generated) field extension $F/k$, we know from \cite[Corollary 6.43]{MField} that there is an isomorphism $[\mathbb A^{2n}_F\setminus 0,\mathbb A^{2n}_F\setminus 0]_{\aone}\cong\mathbf{GW}(F)$ called the motivic Brouwer degree.  By \cite[Remark 2.6]{Fasel11c} (and the fact that the matrix $\mathrm{diag}(t^{-1},t)$ is $\aone$-homotopic to the identity), the motivic Brouwer degree of the map
\[
(a_1,\ldots,a_{2n}) \longmapsto (a_1,t^{-1}a_2,a_3,t^{-1}a_4,\ldots,a_{2n-1},t^{-1}a_{2n})
\]
is precisely $\langle 1\rangle$.
\end{proof}

Theorem \ref{thm:firstnonstablehomotopysheafofsp2n} shows that, for any $n\in\mathbb N$, there is an exact sequence of sheaves of the form
\[
0 \longrightarrow {\mathbf T}'_{2n+2} \longrightarrow \bpi_{2n}^{\aone}(Sp_{2n}) \longrightarrow \mathbf{GW}_{2n+1}^2 \longrightarrow 0.
\]
If we equip $Sp_{2n}$, $Sp_{2n+2}$ and $\mathbb A^{2n}\setminus 0$ with the $\gm$-actions specified above, there is an induced action of $\gm$ on $\bpi_{2n}^{\aone}(Sp_{2n})$ such that the above sequence is an exact sequence of sheaves with $\gm$-action.  The next result is then a consequence of Lemmas \ref{lem:multiplicativeaction} and \ref{lem:trivialaction}.

\begin{cor}
\label{cor:gmactionont2nplus2}
The induced $\gm$-action on ${\mathbf T}'_{2n+2}$ is trivial if $n$ is odd, and restricts to the action given by multiplication on $\mathbf{GW}^2_{2n+1}$ under the map $\bpi_{2n}^{\aone}(Sp_{2n}) \to \mathbf{GW}_{2n+1}^2$.
\end{cor}

\subsubsection*{Bounding the size of $H^3_{\Nis}(X, \mathbf{GW}^2_3(\L))$}
If $X$ is a smooth $k$-scheme (not necessarily a threefold), and $\L$ is a line bundle on $X$, then the $\gm$-action on $\mathbf{GW}^2_3$ described in the previous section can be used to define the sheaf $\mathbf{GW}^2_3(\L)$. In this section, we analyze the Gersten (or Rost-Schmid) resolution of $\mathbf{GW}^2_3(\L)$ in more detail in order to compute $H^3_{\Nis}(X, \mathbf{GW}^2_3(\L))$.

\begin{rem}
A priori, there is another definition of the sheaf $\mathbf{GW}^2_3(\L)$ that appears in the literature: for any \'etale morphism $u: U \to X$, one can sheafify $U \mapsto GW^2_3(U,u^*\L)$ for the Nisnevich topology.  The resulting sheaf is an unramified sheaf and the description of the Gersten resolution shows that this definition coincides with the one we gave.
\end{rem}

Using Karoubi periodicity, we get an exact sequence of sheaves on $X$
\begin{equation}\label{eqn:Karoubi}
\xymatrix{\K_3^Q\ar[r]^-{H_{3,2}} & \mathbf{GW}_3^2(\L)\ar[r]^-{\eta} & \mathbf{GW}_2^1(\L)\ar[r]^-{f_{2,1}} & \K_2^Q};
\end{equation}
we denote by $\mathbf{A}(\L)$ the image of $H_{3,2}$ and by $\mathbf{B}(\L)$ the image of $\eta$. We thus get an exact sequence
\begin{equation}\label{eqn:exactI}
\xymatrix{0\ar[r] & \mathbf{A}(\L)\ar[r] & \mathbf{GW}_3^2(\L)\ar[r]^-{\eta} & \mathbf{B}(\L)\ar[r] & 0.}
\end{equation}

\begin{lem}\label{lem:contractionsAB}
The epimorphism $H_{3,2}:\K_3^Q\to \mathbf{A}(\L)$ induces an isomorphism $\K^Q_1/2\cong \mathbf{A}_{-2}(\L)$. Moreover, $\mathbf{B}_{-2}(\L)\cong \Z/2$ and the exact sequence
\[
\xymatrix{0\ar[r] & \mathbf{A}(\L)_{-2}\ar[r] & \mathbf{GW}_1^0(\L)\ar[r] & \mathbf{B}(\L)_{-2}\ar[r] & 0}
\]
splits if $\L$ is trivial.
\end{lem}

\begin{proof}
Contracting (\ref{eqn:Karoubi}) two times, we get an exact sequence of sheaves
\[
\xymatrix{\K_1^Q\ar[r]^-{H_{1,0}} & \mathbf{GW}_1^0(\L)\ar[r]^-{\eta} & \mathbf{GW}_0^3(\L)\ar[r]^-{f_{0,3}} & \Z.}
\]
By \cite[Lemma 4.1]{Fasel08c}, the hyperbolic functor $H_{0,3}:\Z\to \mathbf{GW}_0^3(\L)$ induces an isomorphism $\Z/2\cong GW_0^3(k(X),\L)$. On the other hand, $(\mathbf{GW}_0^3(\L))_{-1}=\mathbf{W}_0^3(\L)$ which vanishes on any field by \cite[Proposition 5.2]{Balmer02}. It follows that $\mathbf{GW}_0^3(\L)=\Z/2$ and that $f_{0,3}=0$. Thus $\mathbf{B}(\L)_{-2}=\mathbf{GW}_0^3(\L)\cong\Z/2$. If $\L$ is trivial, it is easy to check that the map $\eta:\mathbf{GW}_1^0\to \Z/2$ coincides with the determinant morphism, which is split.

Using \cite[Corollary 4.7.7]{Bass74}, we see that the hyperbolic map $H_{1,0}:\K_1^Q\to \mathbf{GW}_1^0(\L)$ induces an injective map $k(X)^\times/(k(X)^\times)^2\to GW_1^0(k(X),\L)$. It follows that $H_{1,0}$ induces an injective morphism $\K_1^Q/2\to \mathbf{GW}_1^0(\L)$ with, by definition, image equal to $\mathbf{A}(\L)_{-2}$.
\end{proof}

The exact sequence
\[
\xymatrix{0\ar[r] & \mathbf{A}(\L)\ar[r] & \mathbf{GW}_3^2(\L)\ar[r]^-\eta & \mathbf{B}(\L)\ar[r] & 0}
\]
yields an exact sequence of cohomology groups
\[
\xymatrix{H^2_{\Nis}(X,\mathbf{B}(\L))\ar[r] & H^3_{\Nis}(X,\mathbf{A}(\L))\ar[r] & H^3_{\Nis}(X,\mathbf{GW}_3^2(\L))\ar[r] & H_{\Nis}^3(X,\mathbf{B}(\L))}
\]
and we now turn to the task of understanding the cohomology groups of $\mathbf{A}$ and $\mathbf{B}$. We first introduce some notation.

\begin{notation}
\label{notation:Ch}
We will denote by $Ch^n(X)$ the group $CH^n(X)/2$, where $CH^n(X)$ is the Chow groups of codimension $n$ cycles in $X$.
\end{notation}

\begin{lem}\label{lem:H3XA}
We have $H^3_{\Nis}(X,\mathbf{B}(\L))=0$. Moreover, the hyperbolic morphism $H_{3,2}:\K_3^Q\to \mathbf{A}(\L)$ induces an isomorphism $Ch^3(X)\cong H^3_{\Nis}(X,\K_3^Q/2)\cong H^3_{\Nis}(X,\mathbf{A}(\L))$.
\end{lem}

\begin{proof}
Since $\mathbf{B}(\L)_{-2}\cong \Z/2$, it follows that $\mathbf{B}(\L)_{-3}=0$ and therefore $H^3_{\Nis}(X,\mathbf{B}(\L))=0$. Now the hyperbolic morphism induces an isomorphism $\K_1^Q/2\cong \mathbf{A}(\L)_{-2}$ and therefore an isomorphism $\Z/2\cong \mathbf{A}(\L)_{-3}$. It follows that $Ch^3(X)\cong H^3_{\Nis}(X,\K_3^Q/2)\cong H^3_{\Nis}(X,\mathbf{A}(\L))$.
\end{proof}

We now compute the group $H^2_{\Nis}(X,\mathbf{B}(\L))$.

\begin{lem}
The hyperbolic morphism $H_{2,1}:\K_2^Q\to \mathbf{GW}_2^1(\L)$ induces a morphism $H^\prime_{2,1}:\K_2^Q\to \mathbf{B}(\L)$.
\end{lem}

\begin{proof}
By definition of $\mathbf{B}(\L)$, it suffices to show that the composite
\[
\xymatrix{\K_2^Q\ar[r]^-{H_{2,1}} & \mathbf{GW}_2^1(\L)\ar[r]^-{f_{2,1}} & \K_2^Q}
\]
is trivial. By Matsumoto's theorem, the morphism of sheaves $\mu_2:\K_2^M\to\K_2^Q$ is an isomorphism, and it follows therefore from Lemma \ref{lem:computationfH} that $f_{2,1}H_{2,1}=0$ after evaluating at $k(X)$. Thus $f_{2,1}H_{2,1}=0$.
\end{proof}

\begin{lem}\label{lem:H2XB}
The morphism $H^\prime_{2,1}:\K_2^Q\to \mathbf{B}(\L)$ yields an isomorphism
\[
Ch^2(X)\cong H^2_{\Nis}(X,\K_2^Q/2)\to H^2_{\Nis}(X,\mathbf{B}(\L)).
\]
\end{lem}

\begin{proof}
By Lemma \ref{lem:contractionsAB} (or rather its proof), we know that $(H^\prime_{2,1})_{-2}$ induces an isomorphism $\Z/2\cong \mathbf{B}(\L)_{-2}$. The same lemma shows that $H_{1,0}$ yields a monomorphism $\K_1^Q/2\to \mathbf{GW}_{1}^0(\L)$ whose cokernel is isomorphic to $\Z/2$. This cokernel is identified with the image of $f_{1,0}:\mathbf{GW}_1^0(\L)\to \K_1^Q$. It follows that $(H^\prime_{2,1})_{-1}$ yields an isomorphism $\K_1^Q/2\cong \mathbf{B}(\L)_{-1}$.
\end{proof}

Combining Lemmas \ref{lem:H3XA} and \ref{lem:H2XB}, we obtain the following result, which is a key tool in the proof of Theorem \ref{thm:classificationrank2algclosed}.

\begin{prop}
\label{prop:gw23bound}
If $k$ is a field having characteristic unequal to $2$, $X$ is a smooth $k$-scheme and $\L$ is a line bundle on $X$, then there is an exact sequence of the form
\begin{equation}\label{eqn:bounding}
\xymatrix{Ch^2(X)\ar[r]^-\partial & Ch^3(X)\ar[r] &  H^3_{\Nis}(X,\mathbf{GW}_3^2(\L))\ar[r] &  0.}
\end{equation}
\end{prop}

\subsubsection*{Computing $H^3_{\Nis}(X,\mathbf{GW}_3^2(\L))$}
It is possible to compute the map $\partial:Ch^2(X)\to Ch^3(X)$ of Proposition \ref{prop:gw23bound} explicitly. The discussion of the remainder of this section is devoted to providing an explicit description of this map, which is not strictly necessary in the sequel.

To state our result, recall first from \cite[\S 8]{Brosnan03} that one can define Steenrod squaring operations $Sq^2:Ch^n(X)\to Ch^{n+1}(X)$ for any $n\in\mathbb{N}$ satisfying a number of useful properties. In particular, if $f:Y\to X$ is a proper morphism of smooth connected schemes, the formula \cite[8.10, 8.11, 9.4]{Brosnan03}:
\begin{equation}\label{eqn:steenrod}
Sq^2(f_*[Y])=c_1(\omega_{X/k})f_*([Y])-f_*(c_1(\omega_{Y/k}))
\end{equation}
holds, where $\omega_ {X/k}$ (resp. $\omega_ {Y/k}$) is the canonical sheaf of $X$ over $\Spec k$ (resp. $Y$ over $\Spec k$).  In fact, as Totaro explains \cite[Proof of Theorem 1.1]{Totaro}, this property characterizes $Sq^2$. Given a line bundle $\L$ on $X$, we denote by $c_1(\L)$ its class in $Ch^1(X)$. We can then twist the Steenrod operations to obtain a new operation $Sq^2_{\L}:Ch^n(X)\to Ch^{n+1}(X)$ defined by $Sq^2_{\L}(\alpha)=Sq^2(\alpha)+ c_1(\L) \cdot \alpha$.

\begin{thm}
\label{thm:cohomologyofksp3}
If $X$ is a smooth scheme over a field $k$ having characteristic different from $2$, then there is an exact sequence of the form
\[
\xymatrix{Ch^{2}(X)\ar[r]^-{Sq^2_{\L}} & Ch^3(X)\ar[r] & H^3_{\Nis}(X,\mathbf{GW}_3^2(\L))\ar[r] & 0.}
\]
\end{thm}

To prove the theorem, we first observe that $Ch^2(X)$ is generated by the classes of integral subvarieties of codimension $2$. Let then $Y$ be such a subvariety, and let $C$ be its normalization. The morphism $i:C\to Y\subset X$ is finite. The singular locus of $C$ is of codimension $\geq 2$ and its image $T$ under $i$ (which is closed) is of codimension $\geq 4$. Considering $X\setminus T$ and $C\setminus i^{-1}(T)$ instead of $X$ and $C$, we see that we can suppose that the normalization $C$ of $Y$ is smooth. Since $i$ is finite, \cite[Corollary 5.30]{MField} yields homomorphisms
\[
i_*:H^n_{\Nis}(C,\mathbf{F}(\N\otimes\L)_{-2})\to H^{n+2}_{\Nis}(X,\mathbf{F}(\L))
\]
for any strictly $\aone$-invariant sheaf $\mathbf{F}$ (that is already a contraction) and any $n\in\mathbb{N}$, where $\N=\omega_{C/k}\otimes i^*\omega_{X/k}^\vee$ and $\omega_{C/k}$ (resp. $\omega_{X/k}$) is the canonical sheaf of $C$ (resp. $X$). The exact sequence (\ref{eqn:exactI}) then yields a commutative diagram
\[
\xymatrix{H^0_{\Nis}(C,\mathbf{B}(\N\otimes\L)_{-2})\ar[r]^-{\partial^\prime}\ar[d] & H^1_{\Nis}(C,\mathbf{A}(\N\otimes\L)_{-2})\ar[r]\ar[d] & H^1_{\Nis}(C,\mathbf{GW}_1^0(\N\otimes\L))\ar[d] \\
H^2_{\Nis}(X,\mathbf{B}(\L))\ar[r]\ar@{=}[d] & H^3_{\Nis}(X,\mathbf{A}(\L))\ar[r]\ar@{=}[d] & H^3_{\Nis}(X,\mathbf{GW}_3^2(\L)) \\
Ch^2(X)\ar[r]_-{\partial} & Ch^3(X) & }
\]
Now, Lemma \ref{lem:contractionsAB} shows that $\mathbf{A}(\N\otimes\L)_{-2}\cong \K_1^Q/2$ and $\mathbf{B}(\N\otimes\L)_{-2}\cong \Z/2$. There are therefore canonical identifications $H^0_{\Nis}(C,\mathbf{B}(\N\otimes\L)_{-2})\cong \Z/2$ and $H^1_{\Nis}(C,\mathbf{A}(\N\otimes\L)_{-2})\cong H^1_{\Nis}(C,\K_1^Q/2)\cong Pic(C)/2$. The map $i_*:H^0_{\Nis}(C,\mathbf{B}(\N\otimes\L)_{-2})\to H^2_{\Nis}(X,\mathbf{B}(\L))$ satisfies then $i_*(1)=[Y]$. It follows that it is sufficient to compute $\partial^\prime$ in order to identify $\partial$.

\begin{prop}
Under the previous identifications, the homomorphism
\[
\partial^\prime:\Z/2\cong H^0_{\Nis}(C,\mathbf{B}_{-2}(\N\otimes\L))\to H^1_{\Nis}(C,\mathbf{A}_{-2}(\N\otimes\L))\cong Pic(C)/2
\]
satisfies $\partial^\prime(1)=c_1(\N \tensor \L)$.
\end{prop}

\begin{proof}
We first suppose that $c_1(\N\tensor \L)= 0$ in $Pic(C)/2$. There exists then a line bundle $\mathcal M$ and an isomorphism $\N\otimes\L\cong \mathcal M^{\otimes 2}$. In this case, we have an isomorphism $\mathbf{GW}_1^0(\N\otimes\L)\cong \mathbf{GW}_1^0$ by, e.g., \cite[Theorem 5.1]{PaninWalter1}. The sequence
\[
\xymatrix{0\ar[r] & \mathbf{A}(\N\otimes\L)_{-2}\ar[r] & \mathbf{GW}_1^0\ar[r] & \mathbf{B}(\N\otimes\L)_{-2}\ar[r] & 0}
\]
is then split by Lemma \ref{lem:contractionsAB} and therefore $\partial^\prime(1)=0$.

We can therefore assume that $c_1(\N\tensor \L) \neq 0$ in $Pic(C)/2$.  Observe first that the sequence
\begin{equation}\label{eqn:exactII}
\xymatrix{\Z/2\cong H^0_{\Nis}(C,\mathbf{B}(\N\otimes\L)_{-2})\ar[r]^-{\partial^\prime} & H^1_{\Nis}(C,\mathbf{A}(\N\otimes\L)_{-2})\ar[r] & H^1_{\Nis}(C,\mathbf{GW}_1^0(\N\otimes\L))}
\end{equation}
is exact. It suffices thus to prove that $c_1(\L\tensor \N)$ belongs to the kernel of the map $H^1_{\Nis}(C,\K_1^Q/2)\cong H^1_{\Nis}(C,\mathbf{A}(\N\otimes\L)_{-2})\to H^1_{\Nis}(C,\mathbf{GW}_1^0(\N\otimes\L))$ in order to conclude. The Bloch-Ogus spectral sequence \cite[Proposition 3.9]{BlochOgus} (see also \cite[Theorem 25]{FaselSrinivas} in the context of Grothendieck-Witt groups and \cite[Theorem 5.4]{Quillen73} in the context of $K$-theory) yields homomorphisms
\[
F^1GW_0^0(C,\N\otimes\L)\to H^1_{\Nis}(C,\mathbf{GW}_1^0(\N\otimes\L))
\]
and
\[
\tilde K_0(C)\to Pic(C)
\]
where $F^1GW_0^0(C,\N\otimes\L)$ is the kernel of the localization map
\[
GW_0^0(C,\N\otimes\L)\to GW_0^0(k(C),\N\otimes\L)
\]
and $\tilde K_0(C)$ is the kernel of the rank homomorphism. Observe that $[\N\otimes\L]$ in $Pic(C)$ is the image of $[\N\otimes\L]-[\O_C]$ under the second map.

The hyperbolic functor gives a commutative diagram of the form
\[
\xymatrix{\tilde K_0(C)\ar[r]\ar[d]_{H_{0,0}}  & Pic(C)\ar[d] \\
F^1GW_0^0(C,\N\otimes\L)\ar[r] & H^1_{\Nis}(C,\mathbf{GW}_1^0(\N\otimes\L))},
\]
and the right-hand vertical map is the composite
\[
Pic(C)\to Pic(C)/2\cong H^1_{\Nis}(C,\K_1^Q/2)\cong H^1_{\Nis}(C,\mathbf{A}(\N\otimes\L)_{-2})\to H^1_{\Nis}(C,\mathbf{GW}_1^0(\N\otimes\L)).
\]
If $\mathcal M$ is a line bundle over $C$, then the image of $[\mathcal M]-[\O_C]$ under the left-hand vertical map is
\[
[\mathcal M\oplus (\mathcal M^\vee\otimes \N\otimes\L),h_{\mathcal M}]-[\O_C\oplus (\N\otimes\L),h_{\O_C}],
\]
where $h_{\mathcal M}$ and $h_{\mathcal O_C}$ are the usual hyperbolic forms.

It is straightforward to check that $[\N\otimes\L]-[\O_C]$ vanishes under this map. It follows that $c_1(\L\tensor \N)$ belongs to the kernel of $H^1_{\Nis}(C,\K_1^Q/2)\cong H^1_{\Nis}(C,\mathbf{A}(\L)_{-2})\to H^1_{\Nis}(C,\mathbf{GW}_1^0(\L))$, giving $\partial^\prime(1)= c_1(\L\tensor \N)$.
\end{proof}

To conclude the proof of Theorem \ref{thm:cohomologyofksp3}, it suffices now to observe that the diagram
\[
\xymatrix{\Z/2\ar[r]^-{\partial^\prime}\ar[d] & Pic(C)/2\ar[d] \\
Ch^2(X)\ar[r]_-{Sq^2_\L} & Ch^3(X)}
\]
commutes by the formula (\ref{eqn:steenrod}) defining Steenrod operations and the projection formula.

\section{Vanishing theorems}
\label{section:vanishing}
In this section, we prove a number of cohomological vanishing results that will be used in Section \ref{section:obstructionsandclassification} to provide explicit descriptions of sets of isomorphism classes of vector bundles.

\subsubsection*{Cohomology of $\K^{\MW}_j$ and line bundle twists}
Recall some notation from the beginning of Section \ref{section:aonehomotopystablerange}: $\mathbf{I}^n$ is the unramified sheaf (in the Nisnevich or Zariski topology) of the $n$-th power of the fundamental ideal, $\K^{MW}_n$ is the unramified Milnor-Witt K-theory sheaf.  If $\L$ is a line bundle over a smooth $k$-scheme $X$, $\mathbf{I}^n(\L)$ and $\K^{MW}_n(\L)$ are the corresponding $\L$-twisted sheaves (on the small Nisnevich site of $X$).  Furthermore, recall that if $\ell$ is a prime number, then $cd_{\ell}(k)$ is the smallest integer $n$ (or $\infty$) such that $H^i_{\et}(\Spec k,\F) = 0$ for $i > n$ and all $\ell$-torsion \'etale sheaves $\F$ over $\Spec k$.

\begin{prop}\label{prop:vanishingI}
Let $X$ be a smooth scheme of dimension $d$ over a field $k$ with $\mathrm{cd}_2(k)=r<\infty$ and let $\L$ be a line bundle on $X$. Then the Zariski sheaf $\mathbf{I}^{n}(\L)=0$ for any $n\geq r+d+1$.
\end{prop}

\begin{proof}
By definition of $\mathbf{I}^{n}(\L)$ it is sufficient to prove that $I^n(k(X),\L\otimes k(X))=0$. Choosing a generator of $\L\otimes k(X)$ yields an isomorphism $I^n(k(X))\simeq I^n(k(X),\L\otimes k(X))$ and we can thus suppose that $\L$ is trivial.  Consider the quotient group $\overline I^n(k(X)):=I^n(k(X))/I^{n+1}(k(X))$. The affirmation of the Milnor conjecture yields an isomorphism $\overline I^n(k(X))\simeq H^n_{\mathrm{Gal}}(k(X),\mu_2^{\otimes n})$. The latter is trivial since $\mathrm{cd}_2(k(X))\leq r+d$ by \cite[\S 4.2, Proposition 11]{Serre94}. It follows then from \cite[Korollar 2]{Arason71} that $I^n(k(X))=0$.
\end{proof}

We now prove a yet stronger vanishing statement for ${\mathbf I}^{d+r}(\L)$.

\begin{prop}
\label{prop:vanishingII}
Let $X$ be a smooth affine scheme of dimension $d$ over a field $k$ with $\mathrm{cd}_2(k)=r<\infty$ and let $\L$ be a line bundle on $X$. If $d\geq 1$, then $H^d_{\Nis}(X,{\mathbf I}^{j}(\L))=0$ for any $j\geq d+r$. If $d\geq 2$ then we have $H^{d-1}_{\Nis}(X,{\mathbf I}^{j}(\L))=0$ for any $j\geq d+r$.
\end{prop}

\begin{proof}
We prove the vanishing result for cohomology computed in the Zariski topology since the Gersten resolution implies that Zariski and Nisnevich cohomology coincide. By Proposition \ref{prop:vanishingI}, we are reduced to the case $j=d+r$.  The exact sequence of sheaves on the small Nisnevich site of $X$ \cite[\S 2.1]{Fasel09d}
\[
\xymatrix{0\ar[r] & {\mathbf I}^{d+r+1}(\L)\ar[r] & {\mathbf I}^{d+r}(\L)\ar[r] & \overline{\mathbf I}^{d+r}\ar[r] & 0}
\]
and Proposition \ref{prop:vanishingI} give an isomorphism of sheaves $\mathbf{I}^{d+r}(\L) \cong \overline{\mathbf{I}}^{d+r}$.  Using the affirmation of Milnor's conjecture on quadratic forms \cite{OVV} and Voevodsky's affirmation of the Milnor conjecture on the mod $2$ norm residue homomorphism \cite{VMod2}, there are isomorphisms $\overline{\mathbf{I}}^{d+r} \cong \K^{M}_{d+r}/2 \cong \mathcal{H}^{d+r}_{\et}(\mu_2^{d+r})$, where $\mathcal{H}^{i}_{\et}(\mu_2^{j})$ is the Nisnevich sheaf associated with the presheaf $U \mapsto H^i_{\et}(U,\mu_2^{\tensor j})$.  To establish the result in the statement, we will show that $H^{i}(X,\mathcal{H}^{d+r}_{\et}(\mu_2^{d+r})) = 0$ in the stated range.

To this end, consider the Bloch-Ogus spectral sequence \cite{BlochOgus} abutting to $H^*_{\et}(X,\mu_2^{\tensor d+r})$.  By the assumption on cohomological dimension, the lines resolving $\mathcal{H}^{j}_{\et}(\mu_2^{\tensor d+r})$ vanish for $j \geq d+r+1$ (by the assumption on cohomological dimension).  The analysis of the spectral sequence therefore yields an isomorphism
\[
H^{2d+r}_{\et}(X,\mu_2^{\tensor d+r}) \cong H^{d}_{\Nis}(X,\mathcal{H}^{d+r}_{\et}(\mu_2^{\tensor d+r}))
\]
and a surjection
\[
H^{2d+r-1}_{\et}(X,\mu_2^{\tensor d+r}) \longrightrightarrow H^{d-1}_{\Nis}(X,\mathcal{H}^{d+r}_{\et}(\mu_2^{\tensor d+r})).
\]
Now, since $X$ is affine, $H^i_{\et}(X,\mu_2^{\tensor q})=0$ for $i\geq d+r+1$ (\cite[Chapter VI, Theorem 7.2]{Milne} and (\cite[Chapter III, Theorem 2.20]{Milne}; see also \cite[Remark 2.21(b)]{Milne})).
\end{proof}

\begin{cor}
\label{cor:milnorwittmilnorcohomologycomparison}
If $k$ is a quadratically closed field, $X$ is a smooth affine $k$-scheme of dimension $d\geq 2$, and $\L$ is a line bundle on $X$, for any pair of integers $i,j\geq d-1$, there are isomorphisms
\[
H^{i}_{\Nis}(X,\K^{\MW}_j(\L)) \isomto H^i_{\Nis}(X,\K^M_j).
\]
\end{cor}

\begin{proof}
Use the long exact sequence in cohomology associated with the short exact sequence of sheaves
\[
0 \longrightarrow {\mathbf I}^{j+1}(\L) \longrightarrow \K^{\MW}_j(\L) \longrightarrow \K^M_j \longrightarrow 0
\]
and Proposition \ref{prop:vanishingII}.
\end{proof}

\subsubsection*{A vanishing result for cohomology of $\K^M_n/m$}
\begin{prop}
\label{prop:torsionmilnorKtheoryvanishing}
Let $X$ be a smooth variety of dimension $d$ over a field $k$. If there exists an integer $m > 0$ such that for any closed point $x\in X$ the group $k(x)^\times$ is $m$-divisible, then $H^{d}_{\Nis}(X,\K^M_{d+1}/m) = 0$.
\end{prop}

\begin{proof}
The Gersten resolution for the sheaf $\K^M_{d+1}/m$ gives an exact sequence of the form
\[
\bigoplus_{x \in X^{(d)}}(\K^M_{d+1}/m)_{-d}(k(x)) \longrightarrow H^d_{\Nis}(X,\K^M_{d+1}/m) \longrightarrow 0.
\]
Furthermore $(\K^M_{d+1}/m)_{-d} = \K^M_1/m$, and $\K^M_1/m(k(x)) = k(x)^\times/(k(x)^\times)^m = 0$ by assumption.
\end{proof}

\section{Obstruction theory and classification results for vector bundles}
\label{section:obstructionsandclassification}
In this section, we begin by reviewing aspects of obstruction theory involving the Postnikov tower in $\aone$-homotopy theory.  We combine the results of the previous sections with obstruction theory for the Postnikov tower of $BGL_n$ to obtain information about vector bundles.  The main results of this section are Theorems \ref{thm:classificationrank2algclosed} and Corollary \ref{cor:stableisomorphismimpliesisomorphism}.

\subsubsection*{The Postnikov tower in $\aone$-homotopy theory}
If ${\mathscr G}$ is a (Nisnevich) sheaf of groups, and ${\mathbf A}$ is a (Nisnevich) sheaf of abelian groups on which $\mathscr{G}$ acts, there is an induced action of ${\mathscr G}$ on the Eilenberg-Mac Lane space $K({\mathbf A},n)$ that fixes the base-point.  In that case, we set $K^{\mathscr G}(\mathbf A,n) := E{\mathscr G} \times^{\mathscr G} K({\mathbf A},n)$. The projection onto the first factor defines a morphism $K^{\mathscr G}({\mathbf A},n) \to B{\mathscr G}$ that is split by the inclusion of the base-point.

Just as simplicial homotopy classes of maps $[{\mathscr X},K({\mathbf A},n)]_{s}$ are in bijection with elements of $H^n_{\Nis}({\mathscr X},{\mathbf A})$, there is a corresponding classification theorem in this ``twisted" setting.  A map ${\mathscr X} \to K^{{\mathscr G}}({\mathbf A},n)$ gives, by composition, a morphism $\mathscr{X} \to B{\mathscr G}$, which yields a $\mathscr{G}$-torsor ${\mathscr P} \to \mathscr{X}$ by pullback.  Then, the morphism of the previous sentence can be interpreted as a $\mathscr{G}$-equivariant map ${\mathscr P} \to K({\mathbf A},n)$, i.e., a $\mathscr{G}$-equivariant degree $n$ cohomology class on ${\mathscr X}$ with coefficients in $\mathbf{A}$.  The following result summarizes the form of the Postnikov tower we will use; this result is collated from a collection of sources including \cite[Chapter VI.5]{GoerssJardine}, \cite{MV} and \cite[Appendix B]{MField}.

\begin{thm}
\label{thm:postnikovtower}
If $({\mathscr Y},y)$ is any pointed $\aone$-connected space, then there are a sequence of pointed $\aone$-connected spaces $({\mathscr Y}^{(i)},y)$, morphisms $p_i: {\mathscr Y} \to {\mathscr Y}^{(i)}$, and morphisms $f_i: {\mathscr Y}^{(i+1)} \to {\mathscr Y}^{(i)}$ such that
\begin{itemize}
\item[i)] ${\mathscr Y}^{(i)}$ has the property that $\bpi_j^{\aone}({\mathscr Y}^{(i)}) = 0 $ for $j > i$,
\item[ii)] the morphism $p_i$ induces an isomorphism on $\aone$-homotopy sheaves in degree $\leq i$,
\item[iii)] the morphism $f_i$ is an $\aone$-fibration, and the $\aone$-homotopy fiber of $f_i$ is a $K(\bpi_{i+1}^{\aone}({\mathscr Y}),i+1)$,
\item[iv)] the induced morphism ${\mathscr Y} \to \operatorname{holim}_i {\mathscr Y}^{(i)}$ is an $\aone$-weak equivalence.
\end{itemize}
Furthermore, $f_i$ is a {\em twisted $\aone$-principal fibration}, i.e., there is a unique (up to $\aone$-homotopy) morphism
\[
k_{i+1}: \mathscr{Y}^{(i)} \longrightarrow K^{\bpi_1^{\aone}(\mathscr{Y})}(\bpi_{i+1}^{\aone}(\mathscr{Y}),i+2)
\]
called a $k$-invariant sitting in an $\aone$-homotopy pullback square of the form
\[
\xymatrix{
\mathscr{Y}^{(i+1)} \ar[r]\ar[d] & B\bpi_1^{\aone}(\mathscr{Y}) \ar[d] \\
\mathscr{Y}^{(i)} \ar[r]^-{k_{i+1}} & K^{\bpi_1^{\aone}(\mathscr{Y})}(\bpi_{i+1}^{\aone}(\mathscr{Y}),i+2),
}
\]
where the action of $\bpi_1^{\aone}({\mathscr Y})$ on the higher $\aone$-homotopy sheaves is the usual conjugation action induced by change of base-points.
\end{thm}

Let $(\mathscr{X},x)$ be a pointed space.  With notation as in Theorem \ref{thm:postnikovtower}, if $g^{(i)}: \mathscr{X} \to \mathscr{Y}^{(i)}$ is a map, then $g^{(i)}$ lifts to a map $g^{(i+1)}:\mathscr{X} \to \mathscr{Y}^{(i+1)}$ if and only if the composite
\[
\mathscr{X} \stackrel{g^{(i)}}{\longrightarrow} \mathscr Y^{(i)} \longrightarrow K^{\bpi_1^{\aone}({\mathscr Y})}(\bpi_{i+1}^{\aone}({\mathscr Y}),i+2)
\]
lifts to a map $\mathscr{X}\to B\bpi_1^{\aone}(\mathscr{Y})$.

Since each stage of the Postnikov tower is a twisted $\aone$-principal fibration with fiber an Eilenberg-Mac Lane space, the set of lifts of a given $g^{(i)}$ admits a cohomological description.  The space $\mathscr{Y}^{(0)}$ is $\aone$-contractible by assumption, so any pointed map $\mathscr{X} \to \mathscr{Y}^{(0)}$ is $\aone$-homotopy equivalent to the constant map $\mathscr{X} \to y \hookrightarrow \mathscr{Y}^{(0)}$.  There is no obstruction to lifting the constant map to the first stage of the Postnikov tower to obtain a map $\xi: \mathscr{X} \to \mathscr{Y}^{(1)} = B\bpi_1^{\aone}(\mathscr{Y})$, and such a map yields a $\bpi_1^{\aone}({\mathscr Y})$-torsor $\mathscr{P}$ on $\mathscr{X}$.  We will describe lifts to higher stages of the Postnikov tower with $\xi$ fixed.

Given a map $g^{(i)} : \mathscr{X} \to \mathscr{Y}^{(i)}$ for which the obstruction to lifting vanishes, using the homotopy cartesian square that appears in Theorem \ref{thm:postnikovtower}, we can provide a description of the set of lifts of $g^{(i)}$.  Indeed, Theorem \ref{thm:postnikovtower} guarantees that $\mathscr{Y}^{(i+1)} \to \mathscr{Y}^{(i)}$ is an $\aone$-fibration so any lift of $g^{(i)}$ can, after choosing appropriate fibrant models of the $\mathscr{Y}^{(i)}$, be assumed to be given by an actual morphism $\mathscr{X} \to \mathscr{Y}^{(i+1)}$.  Specifying a lift of the morphism $g^{(i)}$ is then equivalent to specifying a section of the map
\[
\mathscr{X} \times_{g^{(i)},\mathscr{Y}^{(i)},f_i} \mathscr{Y}^{(i+1)} \longrightarrow \mathscr{X}.
\]
Since $\mathscr{Y}^{(i+1)}$ is the homotopy fiber product of $\mathscr{Y}^{(i)}$ and $B\bpi_1^{\aone}(\mathscr{Y})$ over $K^{\bpi_1^{\aone}({\mathscr Y})}(\bpi_{i+1}^{\aone}({\mathscr Y}),i+2)$, there is an $\aone$-weak equivalence
\[
\mathscr{X} \times_{g^{(i)},\mathscr{Y}^{(i)},f_i} \mathscr{Y}^{(i+1)}  \cong \mathscr{X} \times_{K^{\bpi_1^{\aone}({\mathscr Y})}(\bpi_{i+1}^{\aone}({\mathscr Y}),i+2)} B\bpi_1^{\aone}(\mathscr{Y})
\]
(on the right, we mean homotopy fiber product). 

Now, existence of a lift of $g^{(i)}$ is, as observed above, equivalent to the composite map $\mathscr{X} \to \mathscr{Y}^{(i)} \to K^{\bpi_1^{\aone}({\mathscr Y})}(\bpi_{i+1}^{\aone}({\mathscr Y}),i+2)$ factoring through $B\bpi_1^{\aone}(\mathscr{Y})$.  In fact, using the choices of the previous paragraph, the map $\mathscr{X} \to K^{\bpi_1^{\aone}({\mathscr Y})}(\bpi_{i+1}^{\aone}({\mathscr Y}),i+2)$ appearing in the fiber product above factors through the map $\xi: \mathscr{X} \to B\bpi_1^{\aone}(\mathscr{Y})$ and there is an induced $\aone$-weak equivalence
\[
\mathscr{X} \times_{K^{\bpi_1^{\aone}({\mathscr Y})}(\bpi_{i+1}^{\aone}({\mathscr Y}),i+2)} B\bpi_1^{\aone}(\mathscr{Y}) \cong \mathscr{X} \times_{\xi,B\bpi_1^{\aone}(\mathscr{Y})} (B\bpi_1^{\aone}(\mathscr{Y}) \times_{K^{\bpi_1^{\aone}({\mathscr Y})}(\bpi_{i+1}^{\aone}({\mathscr Y}),i+2)} B\bpi_1^{\aone}(\mathscr{Y})), 
\]
where, again, both sides are homotopy fiber products.  The homotopy fiber product 
\[
B\bpi_1^{\aone}(\mathscr{Y}) \times_{K^{\bpi_1^{\aone}({\mathscr Y})}(\bpi_{i+1}^{\aone}({\mathscr Y}),i+2)} B\bpi_1^{\aone}(\mathscr{Y})
\]
is precisely the space of fiberwise (simplicial) loops in $K^{\bpi_1^{\aone}({\mathscr Y})}(\bpi_{i+1}^{\aone}({\mathscr Y}),i+2)$, which is identified with $K^{\bpi_1^{\aone}({\mathscr Y})}(\bpi_{i+1}^{\aone}({\mathscr Y}),i+1)$ (see, e.g., \cite[\S 2.4-5]{Robinson} for a related discussion).  Combining these equivalences, we see that the space of lifts of a given $g^{(i)}$ is a quotient of the space of $\aone$-homotopy classes of sections of the projection map
\[
\mathscr{X} \times_{\xi,B\bpi_1^{\aone}(\mathscr{Y})} K^{\bpi_1^{\aone}({\mathscr Y})}(\bpi_{i+1}^{\aone}({\mathscr Y}),i+1) \longrightarrow \mathscr{X},
\]
and this space of sections is precisely $H^{i+1}(\mathscr{X},\bpi_{i+1}^{\aone}(\mathscr{Y})(\xi))$.

\begin{prop}
\label{prop:passingtofinitecoskeleta}
If $X$ is a smooth algebraic variety over a field $k$ of dimension $\leq d$, $(\mathcal{Y},y)$ is any pointed $\aone$-connected space and $i\geq 1$, then the morphism $\mathscr{Y} \to \mathscr{Y}^{(i)}$ induces a map
\[
[X,\mathscr{Y}]_{\aone} \longrightarrow [X,\mathscr{Y}^{(i)}]_{\aone}
\]
that is surjective if $i \geq d-1$, and bijective if $i \geq d$. If moreover $H^d_{\Nis}(X,\bpi_d^{\aone}(\mathscr{Y})(\xi)) = 0$ for any $\bpi_1^{\aone}(\mathscr{Y})$-torsor $\xi$ obtained as the composite $[X,\mathscr{Y}^{(i)}]_{\aone}\to [X,\mathscr{Y}^{(1)}]_{\aone}=[X,B\bpi_1^{\aone}(\mathscr{Y})]_{\aone}$, then the above map is bijective also for $i = d-1$.
\end{prop}

\begin{proof}
We use the Postnikov tower, which actually requires us to consider pointed maps.  However, by adjoining a disjoint base-point to $X$ and observing that there is a canonical bijection $[X_+,(\mathscr{Z},z)]_{\aone} \cong [X,\mathscr{Z}]_{\aone}$ for any pointed space $\mathscr{Z}$, we can make statements about free homotopy classes.

Given an element of $[X,\mathscr{Y}^{(i)}]_{\aone}$, the obstruction to lifting it to $[X,\mathscr{Y}^{(i+1)}]$ is given by the composite map
\[
X \longrightarrow \mathscr{Y}^{(i)} \longrightarrow K^{\bpi_1^{\aone}({\mathscr Y})}(\bpi_{i+1}^{\aone}(\mathscr{Y}),i+2).
\]
This composite is an element of
\[
[X,K^{\bpi_1^{\aone}({\mathscr Y})}(\bpi_{i+1}^{\aone}(\mathscr{Y}),i+2)]_{\aone} \cong H^{i+2}_{\bpi_1^{\aone}({\mathscr Y})}(X,\bpi_{i+1}^{\aone}(\mathscr{Y})),
\]
which vanishes for $i \geq d-1$ because the cohomological dimension of $X$ is $\leq d$.  Assuming this vanishing, we know that the set of lifts is a quotient of $H^d_{\Nis}(X,\bpi_d^{\aone}(\mathscr{Y})(\xi))$.  The latter group vanishes if $i \geq d$, and the results follow immediately.
\end{proof}

\subsubsection*{The first two stages of the Postnikov tower of $BGL_2$}
We now apply the discussion above in the case where $\mathscr{Y} = BGL_n$ with its canonical basepoint.  We saw that $\bpi_1^{\aone}(BGL_n) \cong \gm$, independent of $n$.  The space $B\gm$ has exactly one non-zero homotopy sheaf in degree $1$, and this homotopy sheaf is isomorphic to $\gm$.  The determinant map $GL_n \to \gm$ yields a morphism $BGL_n \to B\gm$ that allows us to identify $B\gm$ as the first stage of the $\aone$-Postnikov tower of $BGL_n$ (see, e.g., \cite[Lemma 4.1]{AsokPi1} and the preceding discussion).

To describe the second stage of the $\aone$-Postnikov tower of $BGL_n$, we need to understand the action of $\bpi_1^{\aone}(BGL_n) = \gm$ on $\bpi_2^{\aone}(BGL_n)$.  The homomorphism $SL_n \to GL_n$ induces a map $BSL_n \to BGL_n$.  Since $BSL_n$ is $\aone$-$1$-connected (this follows \cite[Theorem 7.20]{MField}, but see, e.g., \cite[Theorem 3.2]{AsokFaselSpheres} for this statement), it follows that the $\aone$-universal cover of $BGL_n$ has the $\aone$-homotopy type of $BSL_n$.  If $G$ is a sheaf of groups, and we write $EG$ for the Cech simplicial scheme associated with the structure map $G \to \Spec k$, then $EG$ is simplicially contractible.  The inclusion $SL_n \hookrightarrow GL_n$ determines maps $BSL_n = ESL_n/SL_n \to EGL_n/SL_n$, and $EGL_n/SL_n \to EGL_n/GL_n=BGL_n$.  The map $BSL_n \to EGL_n/SL_n$ is a simplicial weak equivalence, while the map $EGL_n/SL_n \to BGL_n$ is a $\gm$-torsor and thus an $\aone$-covering space by means of the identification $GL_n/SL_n \isomt \gm$ given by the determinant and \cite[Lemma 7.5.1]{MField}.  In particular, $EGL_n/SL_n \to BGL_n$ is a model of the $\aone$-universal cover of $BGL_n$ by \cite[Theorem 7.8]{MField}.  It follows that the action of $\gm$ on the higher $\aone$-homotopy sheaves of $BGL_n$ is induced by ``deck transformations" from the standard $\gm$-action on $EGL_n/SL_n$.

There is a conjugation action of $\gm$ on $SL_n$ induced by the splitting $t \mapsto diag(t,1,\ldots,1)$ of the determinant homomorphism.  This action determines $\gm$-actions on $ESL_n$ and $EGL_n$ such that the map $ESL_n \to EGL_n$ is $\gm$-equivariant.  Thus, there is an $\gm$-equivariant map $BSL_n \to EGL_n/SL_n$.  It is straightforward to check that the $\gm$-action on $EGL_n/SL_n$ coming from the identification of $EGL_n/SL_n \to BGL_n$ as a $\gm$-torsor coincides with the action described in the previous paragraph.  In other words, the action of $\gm$ by ``deck transformations" on the higher $\aone$-homotopy sheaves of $BGL_n$ comes from the conjugation action of $\gm$ on $SL_n$.

When $n = 2$, we can identify $SL_2 = Sp_2$.  Combining Theorem \ref{thm:stabilizationspn} and \cite[Theorem 6.40]{MField} we see that $\K^{MW}_2 = \mathbf{GW}^2_2$.  Under these identifications, the conjugation action of $\gm$ on $SL_2$ induces, by Lemmas \ref{lem:multiplicativeaction} and \ref{lem:actionongwij}, the multiplication action of $\K^{MW}_0$ on $\K^{MW}_2$.  Given this description of the action, the next result provides an explicit description of maps to the second stage of the Postnikov tower of $BGL_2$.

\begin{prop}
\label{prop:twocoskeleton}
For any smooth scheme $X$, we have $[X,BGL_2^{(2)}]_{\aone} = [X,K^{\gm}(\K^{MW}_2,2)]_{\aone}$, where the $\gm$ action on $\K^{MW}_2$ is that described above, and an element of $[X,K^{\gm}(\K^{MW}_2,2)]_{\aone}$ corresponds to a pair $(\xi,\alpha)$, where $\xi: X \to B\gm$ corresponds to a line bundle $\L$ on $X$, and $\alpha \in H^2_{\Nis}(X,\K^{MW}_2(\L))$.
\end{prop}

\begin{proof}
The composite map $X \to K^{\gm}(\K^{MW}_2,2) \to B\gm$ yields a line bundle $\xi: \L \to X$ on $X$.  Write $\L^{\circ}$ for the complement of the zero section of $\L$.  The pullback of the universal $\gm$-torsor along the map classifying $\xi$ is $\L^{\circ}$. As explained above, the map $X \to K^{\gm}(\K^{MW}_2,2)$ can then be identified as a morphism $\L^{\circ} \to K(\K^{MW}_2,2)$ that is $\gm$-equivariant.  Taking the product with the identity, one obtains a $\gm$-equivariant map $\L^{\circ} \to K(\K^{MW}_2,2) \times \L^{\circ}$ that, after forming the quotient by the $\gm$-action, determines a map $X \to K(\K^{MW}_2,2) \times^{\gm} \L^{\circ}$. Now \cite[Lemma B.15]{MField} provides a canonical isomorphism $K(\K^{MW}_2,2) \times^{\gm} \L^{\circ}\cong K(\K^{MW}_2 \times^{\gm} \L^{\circ})$.  As described just prior to the statement of the proposition, the action of $\gm$ on $\K_2^{MW}$ is induced by the multiplication action of $\K^{MW}_0$.  It follows that the sheaf $\K^{MW}_2 \times^{\gm} \L^{\circ}$ is precisely the sheaf $\K^{MW}_2(\L)$ described at the beginning of Section \ref{section:aonehomotopystablerange}.
\end{proof}

\begin{rem}
\label{rem:actiononpi2bgln}
For $n \geq 3$, the induced action of $\gm$ on $\bpi_2^{\aone}(BGL_n) = \K^M_2$ can be seen to be trivial by a more straightforward argument.  In that case, the action of $\gm$ factors through $\hom(\K^M_2,\K^M_2)$, which can be identified with the constant sheaf $\Z$.  Over any algebraically closed extension $L/k$, this action is necessarily trivial because $\gm(L)$ divisible.
\end{rem}

\subsubsection*{Obstruction theory and rank $2$ bundles on a threefold}
From the discussion of the previous section, the $\gm$-action on $\bpi_3^{\aone}(BGL_2)$ is that described in Corollary \ref{cor:gmactionont2nplus2}.  In particular, recall that the $\gm$-action on $\mathbf{T}'_4$ is trivial, while the $\gm$-action on $\mathbf{GW}^2_3$ is the standard action induced by multiplication by $\mathbf{GW}^0_0$.  Repeating the argument of Proposition \ref{prop:twocoskeleton}, we deduce the following result.

\begin{prop}
\label{prop:equivariantsheavescoincide}

If $\L \to X$ is a line bundle whose associated $\gm$-torsor is classified by an element $\xi \in [X,B\gm]_{\aone}$, then there are canonical isomorphisms
\[
\begin{split}
H^3_{\Nis}(X,\mathbf{T}'_4(\xi)) &\cong H^3_{\gm}(\L^{\circ},\mathbf{T}'_4) \cong H^3_{\Nis}(X,\mathbf{T}'_4), \text{ and } \\ H^3_{\Nis}(X,\mathbf{GW}^2_3(\xi)) &\cong H^3_{\gm}(\L^{\circ},\mathbf{GW}^2_3) \cong H^3_{\Nis}(X,\mathbf{GW}^2_3(\L)).
\end{split}
\]
\end{prop}

\begin{thm}
\label{thm:classificationrank2algclosed}
If $k$ is an algebraically closed field having characteristic unequal to $2$, and $X$ is a smooth affine $3$-fold, the map sending a vector bundle of rank $2$ to its Chern classes determines a bijection between the pointed set of isomorphism classes of rank $2$ vector bundles on $X$ and $CH^1(X) \times CH^2(X)$.
\end{thm}

\begin{proof}
By Proposition \ref{prop:twocoskeleton}, we know that $[X,BGL_2^{(2)}]$ consists of pairs $(\L,\alpha)$ consisting of a line bundle $\L \in Pic(X)$ and an element $\alpha \in H^2_{\Nis}(X,\K^{MW}_2(\L))$.
By Corollary \ref{cor:milnorwittmilnorcohomologycomparison}, the canonical morphism
\[
H^2_{\Nis}(X,\K^{MW}_2(\L)) \longrightarrow H^2_{\Nis}(X,\K^M_2)
\]
is an isomorphism.

The map $[X,BGL_2]_{\aone} \to [X,B\gm]_{\aone}$ induced by the first stage of the Postnikov tower for $BGL_2$ sends the class of a vector bundle to its determinant, i.e., its first Chern class.  To identify the class $\alpha$, observe that, by \cite[Remark 7.22]{MField}, the composite map $BGL_2 \to K^{\gm}(\K^{MW}_2,2) \to K(\K^M_2,2)$ is given by the composition $BGL_2 \to BGL_{\infty} \stackrel{c_2}{\to}K(\K^M_2,2)$, where $c_2$ is the second Chern class, considered as a morphism in $\hop{k}$.  Therefore, the elements described at the end of the previous paragraph are precisely the first and second Chern classes.

Applying Proposition \ref{prop:passingtofinitecoskeleta} to $\mathscr{Y} = BGL_2$, we see that the map
\[
[X,BGL_2]_{\aone} \longrightarrow [X,BGL_2^{(2)}]_{\aone}
\]
is surjective.  To conclude the proof, it suffices to establish that the stronger hypothesis of Proposition \ref{prop:passingtofinitecoskeleta} is satisfied, i.e., that $H^3_{\Nis}(X,\bpi_3^{\aone}(BGL_2)(\xi)) = 0$ for $\xi$ the $\gm$-torsor corresponding to the line bundle $\L$ above.

Since $\bpi_3^{\aone}(BGL_2) \cong \bpi_2^{\aone}(SL_2)$ was computed in Theorem \ref{thm:firstnonstablehomotopysheafofsp2n}, we know that $\bpi_2^{\aone}(SL_2)$ is an extension of $\mathbf{GW}_3^2$ by $\mathbf{T}'_4$ which is itself an extension of $\mathbf{S}'_4$ by a quotient of $\mathbf{I}^5$.  Furthermore, $\mathbf{S}'_4$ admits an epimorphism from $\K^M_4/12$.  From this description and Proposition \ref{prop:equivariantsheavescoincide}, the group $H^3_{\Nis}(X,\bpi_3^{\aone}(BGL_2)(\xi))$ fits into an exact sequence of the form
\[
H^3_{\Nis}(X,\mathbf{T}'_4) \longrightarrow H^3_{\Nis}(X,\bpi_3^{\aone}(BGL_2)(\xi)) \longrightarrow H^3_{\Nis}(X,\mathbf{GW}^2_3(\L)).
\]
It suffices then to prove that the groups $H^3_{\Nis}(X,{\mathbf T}^\prime_4)$ and $H^3_{\Nis}(X,\mathbf{GW}_3^2(\L))$ vanish for any $\L \in Pic(X)$.

We now use the assumption that $k$ is algebraically closed.  By Proposition \ref{prop:vanishingI}, we know that the sheaf $\mathbf{I}^5$ restricted to $X$ is just the trivial sheaf so the sheaves $\mathbf{T}'_4$ and $\mathbf{S}'_4$ restricted to $X$ are isomorphic.  For reasons of cohomological dimension, the epimorphism $\K^M_4/12 \to \mathbf{S}'_4$ induces a surjective map $H^3_{\Nis}(X,\K^M_4/12) \to H^3_{\Nis}(X,{\mathbf S}^\prime_4)$, and $H^3_{\Nis}(X,\K^M_4/12)$ vanishes by Proposition \ref{prop:torsionmilnorKtheoryvanishing}. Thus, $H^3_{\Nis}(X,{\mathbf T}^\prime_4)=0$. Now, by Proposition \ref{prop:gw23bound}, $H^3_{\Nis}(X,\mathbf{GW}_3^2(\L))$ is a quotient of $Ch^3(X)$.  Because $X$ is affine of dimension $3$, $CH^3(X)$ is divisible, and thus $Ch^3(X)$ is trivial (in fact, $CH^3(X)$ is uniquely divisible by \cite{SrinivasRoitman}).
\end{proof}

\begin{rem}
With more work, it is possible to reprove the result \cite[Theorem 2.1(iii)]{KumarMurthy} of Mohan Kumar and Murthy for fields having characteristic unequal to $2$.  More precisely, given a smooth affine threefold over an algebraically closed field having characteristic unequal $2$, and given arbitrary classes $(c_1,c_2,c_3) \in \prod_{i=1}^3 CH^i(X)$ there exists a unique rank $3$ vector bundle on $X$ with these Chern classes.  One uses obstruction theory in the same way as above.  In this case, the stronger hypothesis of Proposition \ref{prop:passingtofinitecoskeleta} for $\mathscr{Y} = BGL_3$ is not satisfied for $d = 3$.  Nevertheless, it is still possible to prove that lifts from the second stage of the Postnikov tower to the third stage of the Postnikov tower are parameterized as a set by $CH^3(X)$; this involves showing that the action of $\bpi_1^{\aone}(BGL_3)$ on $\bpi_i^{\aone}(BGL_3)$ is trivial for $i = 2,3$ (the case $i = 2$ is Remark \ref{rem:actiononpi2bgln}).
\end{rem}

Since stably isomorphic vector bundles have equal Chern classes, the following statement, which is a strengthening of \cite[Theorem 5.4]{Fasel3fold}, is an immediate consequence of Theorem \ref{thm:classificationrank2algclosed}.

\begin{cor}
\label{cor:stableisomorphismimpliesisomorphism}
If $X$ is a smooth affine $3$-fold over an algebraically closed field $k$ having characteristic unequal to $2$, and if $E$ and $E^\prime$ are a pair of stably isomorphic rank $2$ vector bundles over $X$, then $E$ and $E^\prime$ are isomorphic.
\end{cor}

\section{Complex realization and $\aone$-homotopy sheaves}
\label{section:complexrealization}
Recall that if $k$ is a field that admits a complex embedding, then the assignment $X \mapsto X(\cplx)$ sending a smooth $k$-scheme $X$ to $X(\cplx)$ equipped with its usual structure of a complex manifold can be extended to a functor $\hop{k} \to \mathscr{H}_{\bullet}$, where $\mathscr{H}_{\bullet}$ is the usual homotopy category of topological spaces \cite[\S 4]{MV}.  In particular, this functor induces homomorphisms
\[
\bpi_{i,j}^{\aone}(\mathcal{X},x)(\cplx) \longrightarrow \pi_{i+j}(\mathcal{X}(\cplx),x).
\]
By a result of Morel \cite[Theorem 6.13]{MField} (see \cite[Theorem 2.2]{AsokFaselSpheres} for a statement in the form we use), the sheaf appearing on the left hand side can be computed in terms of contractions of $\bpi_i^{\aone}(\mathcal{X},x)$.  We refer the reader to \cite[Lemma 2.7 and Proposition 2.9]{AsokFaselSpheres} for a convenient summary of other facts about contractions we will use here.

The goal of this section is to study this homomorphism for $\mathcal{X} = Sp_{2n}$; the main results are Theorems \ref{thm:complexrealizationgl2case} and \ref{thm:complexrealizationsp2n}.  Along the way, we prove Lemmas \ref{lem:contractionsofsprimeneven} and \ref{lem:contractionsofsprimenodd}, which, in particular, establish non-triviality of the sheaf $\mathbf{T}'_{2n}$ that appears in Theorem \ref{thm:firstnonstablehomotopysheafofsp2n}.  The results of this section are independent of those of Section \ref{section:obstructionsandclassification}.

\subsubsection*{Further computations of contracted homotopy sheaves}
More precise statements regarding the structure of the sheaf ${\mathbf T}'_{2n+2}$ of Theorem \ref{thm:firstnonstablehomotopysheafofsp2n} can be made after repeated contraction.  The following results show that the structure of the sheaf ${\mathbf T}'_{2n+2}$ depends on the parity of $n$.  We refer the reader to \cite[Lemma 2.7 and Proposition 2.9]{AsokFaselSpheres} for a convenient summary of other facts about contractions we will use here.

\begin{lem}
\label{lem:contractionsofsprimeneven}
If $n$ is an even natural number, then the morphism of sheaves
\[
\K^M_{2n+2}/(2(2n+1)!) \longrightarrow {\mathbf S}'_{2n+2}
\]
induces an isomorphism $\K^M_{2}/(2n+1)! \to (\mathbf{S}'_{2n+2})_{-2n}$.
\end{lem}

\begin{proof}
Recall from Section \ref{section:aonehomotopynonstable} that there is a commutative diagram of the form
\[
\xymatrix{\K_{2n+2}^{Sp}\ar[r]^-{\varphi_{2n+2}}\ar[d]_-{f_{2n+2,2}}  & \K_{2n+2}^{MW}\ar[d] \\
\K_{2n+2}^Q\ar[r]_-{\psi_{2n+2}} & \K_{2n+2}^M.}
\]
It follows that the sheaf ${\mathbf S}'_{2n+2}$ is the cokernel of the composite map
\[
\xymatrix{\K^{Sp}_{2n+2}=\mathbf {GW}_{2n+2}^2 \ar[r]^-{f_{2n+2,2}} & \K^Q_{2n+2}\ar[r] & \K^M_{2n+2}.}
\]
Contracting $2n$ times and using Proposition \ref{prop:contractionofgw}, we obtain a composite
\[
\xymatrix@C=3em{\mathbf {GW}_2^{2-2n} \ar[r]^-{f_{2,2-2n}} & \K^Q_2\ar[r] & \K^M_2}
\]
whose cokernel is $({\mathbf S}'_{2n+2})_{-2n}$. We know from \cite[Corollary 3.11]{AsokFaselSpheres} that the cokernel of the morphism $\K^Q_2\to \K^M_2$ is precisely $\K^M_2/((2n+1)!)$ and it suffices to show that $f_{2,2-2n}$ is onto to conclude.

Since $n$ is even, we can identify $\mathbf {GW}_2^{2-2n}=\mathbf {GW}_2^2$ and the forgetful map $f_{2,2}$ is the natural epimorphism $\mathbf{GW}^2_2 \cong \K^{MW}_2\to \K^M_2$.
\end{proof}

\begin{lem}
\label{lem:contractionsofsprimenodd}
If $n$ is an odd natural number, then the morphism of sheaves
\[
\K^M_{2n+2}/(2(2n+1)!)\to {\mathbf S}'_{2n+2}
\]
induces an isomorphism after $j$-fold contraction for any $j \geq 2n+1$.
\end{lem}

\begin{proof}
Arguing as in the previous lemma, we obtain a composite morphism
\[
\xymatrix@C=3em{\mathbf {GW}_1^{1-2n} \ar[r]^-{f_{1,1-2n}} & \K^Q_1\ar[r] & \K^M_1}
\]
whose cokernel is $({\mathbf S}'_{2n+2})_{-2n-1}$. The cokernel of $\K^Q_1\to \K^M_1$ is $\K_1^M/((2n+1)!)$ by \cite[Lemma 2.9]{AsokFaselSpheres} and it remains to show that the image of $\mathbf {GW}_1^{1-2n}\to \K^Q_1$ is $2\K^Q_1$. Since $n$ is odd, we can identify $\mathbf {GW}_1^{1-2n}=\mathbf {GW}_1^3$ by Proposition \ref{prop:contractionofgw}.  Combining \cite[Lemma 2.3]{FaselRaoSwan} and Lemma \ref{lem:computationfH} yields the required statement regarding the image.
\end{proof}

\begin{lem}
\label{lem:d2nplus3isnontrivial}
If $n$ is an odd natural number, there is an isomorphism $\mathbf{D}_{2n+3}\cong \mathbf{W}$. In particular, the sheaf $\mathbf{D}_{2n+3}$ is non-trivial.
\end{lem}

\begin{proof}
Consider the stabilization sequence
\[
\cdots \longrightarrow \K^{Sp}_{2n+2} \longrightarrow \K^{MW}_{2n+2} \longrightarrow \bpi_{2n}^{\aone}(Sp_{2n}) \longrightarrow \K^{Sp}_{2n+1} \longrightarrow 0.
\]
Since $n$ is odd, $(\K^{Sp}_{2n+2})_{-(2n+3)} = (\mathbf{GW}^{2}_{2n+2})_{-(2n+3)} \cong (\mathbf{GW}^2_0)_{-1} \cong (\mathbb{Z})_{-1} = 0$ by Proposition \ref{prop:contractionofgw}.  Similarly, using \cite[Lemma 4.1]{Fasel08c}, $(\K^{Sp}_{2n+1})_{-(2n+3)} = (\mathbf{GW}^2_{2n+1})_{-(2n+3)} \cong (\mathbf{GW}^3_0)_{-2} \cong (\Z/2)_{-2} = 0$.  Furthermore, $(\K^{MW}_{2n+2})_{-(2n+3)} = \K^{MW}_{-1} \cong \mathbf{W}$ by \cite[Proposition 2.9]{AsokFaselSpheres}.  Thus, one concludes that $(\bpi_{2n}^{\aone}(Sp_{2n}))_{-(2n+3)} \cong \mathbf{W}$.

By Theorem \ref{thm:firstnonstablehomotopysheafofsp2n}, and the discussion of the previous paragraph, it follows that $(\mathbf{T}'_{2n+2})_{-(2n+3)}\cong \mathbf{W}$.  Now, since $(\mathbf{S}'_{2n+2})_{-(2n+3)} \cong (\K^M_{2n+2}/(2(2n+1)!))_{-(2n+3)}$ by Lemma \ref{lem:contractionsofsprimenodd}, it follows from \cite[Lemma 2.7]{AsokFaselSpheres} that $(\mathbf{S}'_{2n+2})_{-(2n+3)} = 0$. Therefore, $(\mathbf{D}_{2n+3})_{-(2n+3)} \cong \mathbf{W}$ as well.
\end{proof}

\begin{thm}
\label{thm:complexrealizationgl2case}
The homomorphisms
\[
\begin{split}
\bpi_{2,3}^{\aone}(SL_2)(\cplx) &\longrightarrow \pi_5(SL_2(\cplx)) \cong \pi_5(SU(2)) = \Z/2, \text{ and } \\
\bpi_{2,4}^{\aone}(SL_2)(\cplx) &\longrightarrow \pi_6(SL_2(\cplx)) \cong \pi_6(SU(2)) = \Z/12
\end{split}
\]
are isomorphisms.
\end{thm}

\begin{proof}
Recall the description of $\bpi_2(SL_2)$ from Theorem \ref{thm:firstnonstablehomotopysheafofsp2n}: there is an exact sequence of the form:
\[
\mathbf{I}^5 \longrightarrow \mathbf{T}_4^{\prime} \longrightarrow \mathbf{S}_4^{\prime} \longrightarrow 0.
\]
Contracting this sequence $4$ times, evaluating on $\cplx$ (using the fact that $\mathbf{I}(\cplx) = 0$), and using Lemma \ref{lem:contractionsofsprimenodd} yields an isomorphism $\Z/12 \isomt (\mathbf{S}_4^{\prime})_{-4}(\cplx)$.  We observed that $(\K^{Sp}_3)_{-4} = 0$ in the proof of Lemma \ref{lem:d2nplus3isnontrivial}.  Thus, there is an isomorphism $\Z/12 \isomt \bpi_{2,4}^{\aone}(SL_2)(\cplx)$.

By \cite[Proposition 19.1]{BorelSerre} one knows that the classifying map of the $Sp_2$-bundle $Sp_{4}/Sp_2 \to BSp_2$ provides a generator of $\pi_6(S^3) \cong \pi_7(BSp_2(\cplx))$.  The computation of $\bpi_{2}^{\aone}(SL_2)$ was achieved using the $\aone$-fiber sequence $Sp_2 \to Sp_4 \to Sp_4/Sp_2$, and the isomorphism $\Z/12 \to \bpi_{2,4}^{\aone}(SL_2)(\cplx)$ is induced by the connecting homomorphism of the associated long exact sequence in $\aone$-homotopy sheaves.  Since the fiber sequence $Sp_2 \to Sp_4 \to Sp_4/Sp_2$ is mapped under complex realization to one homotopy equivalent to that considered by Borel-Serre, the result follows.

For the first isomorphism, recall first that $\K^M_1/12(\cplx) = 0$.  Then, using the fact that $\K^{Sp}_3 = \mathbf{GW}^2_3$ observe that $(\mathbf{GW}^2_3)_{-3} = (\mathbf{GW}^{0}_1)_{-1} = \Z/2$ (again, use Lemma \cite[Lemma 4.1]{Fasel08c} and Proposition \ref{prop:contractionofgw}).  Now, by Lemma \ref{lem:contractionsofsprimeneven}, the fact that $\K^M_1/12(\cplx) = 0$, the fact that $\mathbf{I}^2(\cplx) = 0$, and the fact that $(\K^{Sp}_{3})_{-3} = \Z/2$, we see that $\bpi_{2,3}^{\aone}(SL_2)(\cplx) \cong \Z/2$.  Thus, complex realization gives a map $\bpi_{2,3}^{\aone}(SL_2)(\cplx) = \Z/2 \to \Z/2$.  The computation of \cite{Whiteheadpinplus2} shows that (see the proof of \cite[Theorem 15.2]{HuHomotopy} for more details) the generator of $\bpi_5(S^3)$ is obtained as follows: start with the Hopf map $\eta_{\cplx}: S^3 \to S^2$ and consider the composition $\Sigma \eta_{\cplx} \circ \Sigma^2 \eta_{\cplx}$.  Now, there is the algebro-geometric Hopf map $\eta: {\mathbb A}^2 \setminus 0 \to \pone$ (see \cite[\S 6.3 and Example 6.26]{MField}), and taking the $\gm$ and $\pone$-suspensions of this map we obtain:
\[
\begin{split}
\Sigma_{\gm} \eta&: {\pone}^{\wedge 2} \longrightarrow {\mathbb A}^2 \setminus 0 \\
\Sigma_{\pone} \eta&: {\mathbb A}^3 \setminus 0 \longrightarrow {\pone}^{\wedge 2}.
\end{split}
\]
The composite of these two maps has complex realization the generator of the $\pi_5(S^3)$ since the complex realization of $\eta$ is the usual Hopf map $\eta_{\cplx}$.
\end{proof}

\begin{thm}
\label{thm:complexrealizationsp2n}
The homomorphism
\[
\bpi_{2n,2n+2}^{\aone}(Sp_{2n})(\cplx) \longrightarrow \bpi_{4n+2}(Sp_{2n}(\cplx))
\]
induced by complex realization is an isomorphism.
\end{thm}

\begin{proof}
The complex realization of $Sp_{2n}$ is the group $Sp_{2n}(\cplx)$, which is homotopy equivalent to its maximal compact subgroup (denoted $Sp_n$ in \cite{BHarris2}).  It follows from the table on \cite[p.175]{BHarris2} that $\pi_{4n+2}(Sp_{2n}(\cplx))$ is $\Z/(2n+1)!$ if $n$ is even and $\Z/(2(2n+1)!)$ if $n$ is odd.  There is a canonical morphism $\bpi_{2n,2n+2}^{\aone}(Sp_{2n})(\cplx) \to \pi_{4n+2}(Sp_{2n}(\cplx))$.  Since $\mathbf{W}(\cplx) = \Z/2$, in view of Lemmas \ref{lem:contractionsofsprimeneven} and \ref{lem:contractionsofsprimenodd}, it suffices to prove that in either case we can lift a generator.  In each case, the generator is the image of a generator of $\bpi_{2n+1,2n+2}^{\aone}({\mathbb A}^{2n+2} \setminus 0)(\cplx) \cong (\K^{MW}_{2n+2})_{-2n-2}(\cplx) \cong \mathbf{GW}(\cplx) = \Z$.  Since the generator of this group is mapped to the generator of $\pi_{4n+3}(S^{4n+3})$ under complex realization, the result follows by comparison with the proofs in \cite[pp.182-184]{BHarris2}.
\end{proof}

\begin{footnotesize}
\bibliographystyle{alpha}
\bibliography{vectorbundlesonthreefolds}
\end{footnotesize}
\end{document}